\documentclass[nopreprintline,12pt]{elsarticle}

\usepackage{amsthm}
\usepackage{lineno}
\usepackage{latexsym,amscd,amsfonts,enumerate,supertabular}
\usepackage{tabularx}
\usepackage{array}
\usepackage{multicol}
\usepackage{bigstrut}
\usepackage{algpseudocode}
\usepackage{graphicx}
\usepackage{epsfig}
\usepackage{amssymb}
\usepackage[intlimits]{amsmath}
\usepackage{mathrsfs} 
\usepackage{algorithm}

\usepackage{multirow}
\usepackage{hhline}	
\usepackage{makecell}
\usepackage{color}
\usepackage{xcolor}
\usepackage{float}
\usepackage{graphicx,subfigure}
\usepackage{mathabx}

\usepackage{mathrsfs}

\definecolor{darkmagenta}{rgb}{0.55, 0.0, 0.55}

\newcolumntype{R}[1]{>{\raggedleft\let\newline\\\arraybackslash\hspace{0pt}}m{#1}}
\newcolumntype{L}[1]{>{\raggedright\let\newline\\\arraybackslash\hspace{0pt}}m{#1}}
\newcolumntype{C}[1]{>{\centering\let\newline\\\arraybackslash\hspace{0pt}}m{#1}}

\numberwithin{equation}{section}
\newtheorem{lemma}[]{Lemma}
\numberwithin{lemma}{section}
\newtheorem{theorem}[]{Theorem}
\numberwithin{theorem}{section}
\numberwithin{algorithm}{section}
\newtheorem{remark}[]{Remark}
\numberwithin{remark}{section}
\newtheorem{example}[]{Example}
\numberwithin{example}{section}
\newtheorem{definition}[]{Definition}
\numberwithin{definition}{section}
\newtheorem{corollary}[]{Corollary}
\numberwithin{corollary}{section}
\newtheorem{proposition}{Proposition}
\numberwithin{proposition}{section}
\newcommand{\Sph}{\mathbb{S}^{d-1}}
\newcommand{\E}{\mathbb{E}}
\newcommand{\op}{\mathrm{op}}            
\newcommand{\even}{\mathrm{even}}        
\newcommand{\odd}{\mathrm{odd}}          
\renewcommand{\H}{\mathcal{H}}
\newcommand{\PP}{\mathbb{P}}
\newcommand{\Var}{\mathrm{Var}}
\newcommand{\Cov}{\mathrm{Cov}}
\newcommand{\Corr}{\mathrm{Corr}}

\DeclareMathOperator{\tr}{tr}

\begin{document}
	\begin{frontmatter}
		\title{{\bf
				Size–Location Correlation for Set-Valued Processes: Theory, Estimation, and Laws of Large Numbers under $\rho$-Mixing
		}}
		\author[ioit]{Luc T. Tuyen\corref{cor1}}
		\ead{tuyenlt.ioit@gmail.com}

		\cortext[cor1]{Corresponding author}
		\address[ioit]{Department of Data science and Big data analysis, Institute of Information Technology, Vietnam Academy of Science and Technology, Hanoi, Vietnam}
\begin{abstract}
	We propose a variational framework for analyzing dependence structures of
	convex compact random sets based on their support functions.
	The approach relies on the canonical even--odd decomposition on the unit
	sphere, which separates size-related and location-related components and
	induces an exact orthogonality in the sphere $L^2(\sigma)$ space.
	This decomposition yields an additive variance--covariance structure that is
	intrinsic to set-valued data and cannot be recovered from point-based or
	selection-based representations.
	
	Within this framework, we introduce size, location, and total covariance and
	correlation indices for random sets, together with compatible $\rho$-mixing
	coefficients for set-valued processes.
	The resulting dependence measures are geometrically interpretable, invariant
	under translations, and free of degeneracies that arise for centrally
	symmetric sets under classical approaches.
	Weak and strong laws of large numbers are established under weak stationarity,
	providing asymptotic stability of Minkowski averages in the $L^2(\sigma)$
	support-function norm.
	
	The proposed quantities admit natural numerical realizations via directional
	Monte Carlo and spherical designs.
	Applications to interval-valued and convex-valued data, including regression
	with set-valued responses, illustrate how the even--odd decomposition
	disentangles directional location dependence from size effects beyond what can
	be captured by finite-dimensional summaries such as the Steiner point.
\end{abstract}

\begin{keyword}
	\small
	random sets\sep support functions\sep variational analysis\sep
	size--location decomposition\sep covariance and correlation\sep
	set-valued $\rho$-mixing\sep weak stationarity
	
	\MSC[2020] Primary: 60G10\sep 49J53\sep
	Secondary: 60F05\sep 60F15
\end{keyword}

\end{frontmatter}

\section{Introduction}

Set-valued random variables, also known as random sets, arise naturally in
variational and geometric models where uncertainty affects entire regions,
shapes, or feasible sets rather than individual points.
Typical examples include random convex bodies in stochastic geometry,
interval-valued observations in symbolic data analysis, and solution sets of
random variational inequalities.
A standard analytical representation of convex compact random sets is provided
by their support functions, which embed set-valued objects into function spaces
on the unit sphere and allow the use of variational and functional-analytic
tools \cite{Molchanov2005,Schneider2014,Gardner2006}.

A fundamental operation in this setting is averaging.
Minkowski sums and Aumann expectations play a central role in describing mean
behavior and stability of random sets, and their asymptotic properties have
been extensively studied under independence and various weak dependence
assumptions; see, among others,
\cite{artstein1975strong,Molchanov2005,li2013limit,chen2015strong, tuyen2026strong}.
These results are typically formulated in terms of support functions and rely
on the Hilbert-space structure of $L^2(\sigma)$.

Beyond expectations, however, a systematic variational theory of second-order
quantities for random sets remains incomplete.
In particular, classical notions of variance, covariance, and correlation do
not directly extend to the set-valued setting.
Approaches based on selections or Aumann expectations implicitly aggregate the
support function over directions and therefore confound fundamentally distinct
sources of variability.
As a result, they fail to distinguish between size-related fluctuations and
location-related fluctuations and may degenerate in important geometric
situations, such as centrally symmetric random sets whose location component
vanishes identically.

The present work adopts a structural viewpoint based on the canonical
even--odd decomposition of support functions on the unit sphere.
For a convex compact set, this decomposition separates the even part, encoding
size and width information, from the odd part, encoding location information.
Crucially, the two components are exactly orthogonal in the sphere
$L^2(\sigma)$ space.
This orthogonality induces an additive variance--covariance decomposition that
is intrinsic to set-valued data and has no analogue for point-valued
representations.

Within this framework, classical location summaries such as the Steiner point
appear as finite-dimensional projections of the odd component.
While useful for certain purposes, such projections necessarily discard
directional information and all dependence carried by the even component.
The even--odd decomposition therefore provides a strictly richer description
of geometric variability and dependence for random sets.

Building on this variational structure, we introduce size, location, and total
covariance and correlation indices for convex compact random sets.
These quantities are translation invariant, admit a positive semidefinite
covariance structure, and satisfy natural sign-flip and degeneracy properties
that reflect the underlying geometry.
To capture temporal dependence, we further define compatible $\rho$-mixing
coefficients for set-valued processes via maximal correlations of support
projections, extending the classical scalar notion while preserving geometric
interpretability.

Weak and strong laws of large numbers for weakly stationary $\rho$-mixing
sequences then follow as stability results for Minkowski averages in the
$L^2(\sigma)$ support-function norm.
The role of probability theory in this paper is thus primarily instrumental:
it provides asymptotic guarantees for variational quantities that are defined
and interpreted through the even--odd decomposition.

From a practical perspective, the proposed dependence measures admit natural
numerical realizations via directional Monte Carlo integration and spherical
designs.
Simulation studies demonstrate that the resulting estimators clearly separate
directional location dependence from size effects in situations where
Steiner-point--based summaries are uninformative.

\medskip
The main contributions of the paper are summarized as follows.
\begin{itemize}
	\item A variational size--location--total covariance and correlation framework
	for convex compact random sets based on the even--odd decomposition of support
	functions, with exact orthogonality and additive variance structure.
	\item Compatible $\rho$-mixing coefficients for set-valued processes, defined
	through support-function projections and consistent with the classical scalar
	theory.
	\item Weak and strong laws of large numbers in the $L^2(\sigma)$ support-function
	norm for weakly stationary $\rho$-mixing sequences, ensuring asymptotic
	stability of Minkowski averages.
	\item Directional Monte Carlo and spherical-design estimators for the proposed
	dependence measures, together with numerical studies illustrating their
	geometric interpretability.
	\item Applications to interval-valued and convex-valued data, including
	regression models with set-valued responses.
\end{itemize}

The remainder of the paper is organized as follows.
Section~\ref{sec:Notation} introduces notation and background material on support
functions and the even--odd decomposition.
Section~\ref{sec:New_definitions} develops the variational covariance and
correlation framework.
Section~\ref{sec:limits} presents the asymptotic stability results under weak
dependence.
Section~\ref{sec:simulation} reports numerical studies.
Section~\ref{sec:appls} discusses applications.

\section{Notation and Preliminaries}\label{sec:Notation}

\paragraph{Probability framework.}
We work on a fixed probability space $(\Omega,\mathcal{F},\mathbb{P})$.
Let $(\mathbb{S}^{d-1},\mathcal{B},\sigma)$ denote the unit sphere in $\mathbb{R}^d$
endowed with its Borel $\sigma$-algebra and the normalized surface measure
$\sigma(\mathbb{S}^{d-1})=1$.
In the one-dimensional case $d=1$, we have $\mathbb{S}^0=\{-1,1\}$ with the normalized
counting measure $\sigma=\tfrac12(\delta_{-1}+\delta_1)$.
All expectations, variances, and covariances are taken with respect to $\mathbb{P}$
and denoted by $\mathbb{E}$, $\Var$, and $\Cov$.
We freely use the product space
$(\Omega\times\mathbb{S}^{d-1},\,\mathcal{F}\otimes\mathcal{B},\,\mathbb{P}\otimes\sigma)$
and Fubini--Tonelli whenever integrability is ensured.

\paragraph{Random sets and support functions.}
A (convex compact) random set is a measurable mapping
$X:\Omega\to\mathcal{K}_c(\mathbb{R}^d)$
with Effros/Fell measurability.
Its support function is defined by
\[
h_X(\omega,u):=\sup_{x\in X(\omega)}\langle u,x\rangle,
\qquad u\in\mathbb{S}^{d-1}.
\]
The map $(\omega,u)\mapsto h_X(\omega,u)$ is jointly measurable,
square-integrable in $\omega$ for $\sigma$-a.e.\ $u$, and continuous in $u$.
Throughout the paper we assume
\begin{equation}\label{eq:moment2}
	\int_{\mathbb{S}^{d-1}}\mathbb{E}\,|h_X(u)|^2\,d\sigma(u)<\infty,
\end{equation}
(and analogously for other random sets when needed),
which legitimizes exchanging $\mathbb{E}$ and
$\int_{\mathbb{S}^{d-1}}\cdot\,d\sigma$ by Fubini--Tonelli.

\paragraph{Even--odd decomposition.}
For $u\in\mathbb{S}^{d-1}$ define the even and odd parts of the support function by
\[
W_X(u):=\tfrac12\big(h_X(u)+h_X(-u)\big),
\qquad
C_X(u):=\tfrac12\big(h_X(u)-h_X(-u)\big).
\]
Then $W_X$ is even and $C_X$ is odd in $u$, and
$h_X(u)=W_X(u)+C_X(u)$.
We refer to $W_X$ as the \emph{size} component and to $C_X$ as the
\emph{location} component of $X$.
Centering is taken pointwise in $u$ with respect to $\mathbb{P}$:
\[
\widetilde{W}_X(u):=W_X(u)-\mathbb{E}W_X(u),
\qquad
\widetilde{C}_X(u):=C_X(u)-\mathbb{E}C_X(u).
\]

\paragraph{Sphere $L^2(\sigma)$ norm.}
For any square-integrable function $Z:\mathbb{S}^{d-1}\to\mathbb{R}$ we define
\[
\|Z\|_{2,\sigma}^2
:=\int_{\mathbb{S}^{d-1}} Z(u)^2\,d\sigma(u).
\]
This norm is applied in particular to $W_X$, $C_X$, and their centered versions.
We denote by
$\mathcal{H}:=L^2(\mathbb{S}^{d-1},\sigma)$
the corresponding Hilbert space.

\paragraph{Unified profiles.}
To streamline notation, we introduce for $\star\in\{\mathrm{size},\mathrm{loc},\mathrm{tot}\}$
the profiles
\[
\Phi_X^{(\mathrm{size})}(u):=W_X(u),
\qquad
\Phi_X^{(\mathrm{loc})}(u):=C_X(u),
\qquad
\Phi_X^{(\mathrm{tot})}(u):=h_X(u),
\]
together with their centered versions
\[
\widetilde{\Phi}_X^{(\star)}(u)
:=\Phi_X^{(\star)}(u)-\mathbb{E}\big[\Phi_X^{(\star)}(u)\big].
\]
Whenever convenient, we view $\Phi_X^{(\star)}$ (or $\widetilde{\Phi}_X^{(\star)}$)
as an element of the Hilbert space $\mathcal{H}$.
The evaluation at a direction $u\in\mathbb{S}^{d-1}$ is denoted explicitly by
$\Phi_X^{(\star)}(u)$.
Centering is always understood pointwise in $u$.

\paragraph{Partial sums.}
For a sequence of random sets $(X_i)_{i\ge1}$ and $\star\in\{\mathrm{size},\mathrm{loc},\mathrm{tot}\}$,
we write
\[
S_n^{(\star)}(u)
:=\frac1n\sum_{i=1}^n \widetilde{\Phi}_{X_i}^{(\star)}(u),
\qquad
n\ge1.
\]
When interpreted as an element of $\mathcal{H}$,
$S_n^{(\star)}$ denotes the function
$u\mapsto S_n^{(\star)}(u)$.
No ambiguity will arise from this identification.

\paragraph{Convention.}
Throughout the paper, the intended meaning of expressions involving
$\Phi_X^{(\star)}$ is clear from context:
$\Phi_X^{(\star)}(u)$ denotes a real-valued random variable at direction $u$,
while $\|\Phi_X^{(\star)}\|_{2,\sigma}$ and similar expressions refer to the
$L^2(\sigma)$ norm of the corresponding profile.
All expectations and covariances of such profiles are understood pointwise in $u$
and integrated over $\mathbb{S}^{d-1}$ when required.

\section{Dependence measures via size--location decomposition}\label{sec:New_definitions}
\subsection{Definition of covariance, variance, and correlation}\label{subsec:def-covcorr}

\begin{definition}[Size/location/total covariance, variance, correlation]\label{def:1}
	Let $X,Y$ be set-valued random variables satisfying \eqref{eq:moment2}.
	Define
	\[
	\Cov_{\mathrm{size}}(X,Y)
	:=\int_{\mathbb S^{d-1}}\Cov\big(\widetilde W_X(u),\widetilde W_Y(u)\big)\,d\sigma(u),
	\]
	\[
	\Cov_{\mathrm{loc}}(X,Y)
	:=\int_{\mathbb S^{d-1}}\Cov\big(\widetilde C_X(u),\widetilde C_Y(u)\big)\,d\sigma(u),
	\]
	and $\Var_{\mathrm{size}}(X):=\Cov_{\mathrm{size}}(X,X)$, $\Var_{\mathrm{loc}}(X):=\Cov_{\mathrm{loc}}(X,X)$.
	The \emph{total} covariance and variance are
	\[
	\Cov_{\mathrm{tot}}(X,Y):=\Cov_{\mathrm{size}}(X,Y)+\Cov_{\mathrm{loc}}(X,Y),
	\qquad
	\Var_{\mathrm{tot}}(X):=\Var_{\mathrm{size}}(X)+\Var_{\mathrm{loc}}(X).
	\]
	Whenever $\Var_{\star}(X)\Var_{\star}(Y)>0$ for $\star\in\{\mathrm{size},\mathrm{loc},\mathrm{tot}\}$, define
	\[
	\Corr_{\star}(X,Y)
	:=\frac{\Cov_{\star}(X,Y)}{\sqrt{\Var_{\star}(X)\Var_{\star}(Y)}}\in[-1,1].
	\]
\end{definition}

\begin{proposition}[Structural orthogonality and additivity]\label{prop:orth-add}
	For any set-valued random variables $X,Y$ satisfying \eqref{eq:moment2},
	\begin{equation}\label{eq:orthogonality}
		\int_{\mathbb S^{d-1}}\E\big[\widetilde W_X(u)\,\widetilde C_Y(u)\big]\,d\sigma(u)=0.
	\end{equation}
	Consequently,
	\[
	\Cov_{\mathrm{tot}}(X,Y)=\Cov_{\mathrm{size}}(X,Y)+\Cov_{\mathrm{loc}}(X,Y),
	\qquad
	\Var_{\mathrm{tot}}(X)=\Var_{\mathrm{size}}(X)+\Var_{\mathrm{loc}}(X),
	\]
	and the total covariance/variance in Definition~\ref{def:1} involves no cross-terms.
\end{proposition}

\begin{proof}
	Fix $X,Y$ and consider the integrand $u\mapsto \E[\widetilde W_X(u)\widetilde C_Y(u)]$.
	Since $\widetilde W_X$ is even in $u$ and $\widetilde C_Y$ is odd in $u$, their product is odd:
	\[
	\E[\widetilde W_X(-u)\widetilde C_Y(-u)]
	=\E[\widetilde W_X(u)\cdot (-\widetilde C_Y(u))]
	=-\E[\widetilde W_X(u)\widetilde C_Y(u)].
	\]
	Integrating this odd function over the sphere with the symmetric measure $\sigma$ yields \eqref{eq:orthogonality}.
	The additivity identities follow by expanding $\widetilde h_X=\widetilde W_X+\widetilde C_X$ and integrating, with the cross-term vanishing by \eqref{eq:orthogonality}.
\end{proof}

\begin{remark}
	The bound $|\Corr_{\star}(X,Y)|\le 1$ (for $\star\in\{\mathrm{size},\mathrm{loc},\mathrm{tot}\}$ whenever defined)
	will be derived from a positive semidefinite covariance matrix property; see Proposition~\ref{prop:PSD-CS}.
\end{remark}

\subsection{Relation to Steiner point centering}\label{subsec:steiner}
\begin{definition}[Steiner--centered (residual) location component]
	Let $X$ be a random convex compact set in $\mathbb R^d$ with support function $h_X$
	and odd component
	\[
	C_X(u) := \tfrac12\big(h_X(u)-h_X(-u)\big), \qquad u\in\mathbb S^{d-1}.
	\]
	Let $s(X)\in\mathbb R^d$ denote its Steiner point.  
	We define the \emph{Steiner--centered (residual) location component} by
	\[
	C_X^{\mathrm{res}}(u)
	\ :=\ 
	C_X(u)-\langle s(X),u\rangle, 
	\qquad u\in\mathbb S^{d-1}.
	\]
\end{definition}

\begin{remark}[Interpretation and role in dependence analysis]\label{rem:residual-loc}
	The decomposition
	\[
	C_X(u)
	\ =\ 
	\langle s(X),u\rangle
	\ +\ 
	C_X^{\mathrm{res}}(u)
	\]
	separates the odd part of the support function into a \emph{linear component}
	induced by the Steiner point and a \emph{residual odd component} capturing
	higher--order location fluctuations.
	The residual $C_X^{\mathrm{res}}$ belongs to the odd subspace of
	$L^2(\mathbb S^{d-1},\sigma)$, is centered, and removes the linear component
	generated by the Steiner point.
		
	In particular, any correlation or covariance defined via $C_X^{\mathrm{res}}$
	measures location dependence beyond what can be represented by a single
	point--valued summary such as the Steiner point.
	All limit theorems established for the odd component $C_X$
	(WLLN, FCLT, LIL) remain valid for $C_X^{\mathrm{res}}$,
	since the subtraction of the linear term $\langle s(X),u\rangle$ does not affect
	stationarity, mixing, or moment conditions.
\end{remark}

\begin{remark}[Steiner point as a projection of the odd component]\label{rem:steiner-proj}
	A classical location descriptor for a convex compact set is its Steiner point $s(K)\in\mathbb{R}^d$,
	which is translation-equivariant and rotation-invariant.
	It can be represented (up to a dimension-dependent constant) as a linear functional of the support function,
	\[
	s(K)\ \propto\ \int_{\mathbb{S}^{d-1}} u\,h_K(u)\,d\sigma(u).
	\]
	Since $u$ is odd and $h_K=W_K+C_K$ with $W_K$ even and $C_K$ odd, the even part does not contribute:
	\[
	\int_{\mathbb{S}^{d-1}}u\,W_K(u)\,d\sigma(u)=0,
	\qquad
	s(K)\ \propto\ \int_{\mathbb{S}^{d-1}}u\,C_K(u)\,d\sigma(u).
	\]
	Thus, the Steiner point can be viewed as a finite-dimensional \emph{projection} of the odd component.
	In contrast, our location profile $u\mapsto C_K(u)$ retains full directional
	information in $L^2(\sigma)$, which is essential for detecting directional
	location dependence, as illustrated in the simulation study.
	
\end{remark}

\begin{example}[Directional dependence invisible to the Steiner point]\label{ex:steiner-fail}
	Let $d\ge 2$ and let $A\in\mathcal{K}_c(\mathbb{R}^d)$ be a non-centrally symmetric convex body whose Steiner point is at the origin, $s(A)=0$.
	(Such an $A$ is obtained by translating any non-centrally symmetric body by minus its Steiner point.)
	Let $(\varepsilon_n)_{n\in\mathbb{Z}}$ be a stationary real-valued process with $\E[\varepsilon_n]=0$
	and $\varepsilon_n\ge 0$ a.s.\ (so that Minkowski scaling is well-defined).
	Define a set-valued time series by
	\[
	K_n:=K_0+\varepsilon_n A,
	\]
	where $K_0$ is a fixed centrally symmetric convex body with $s(K_0)=0$.
	
	By Minkowski additivity of the Steiner point and $s(A)=s(K_0)=0$, we have $s(K_n)=0$ for all $n$.
	Hence any dependence analysis based solely on the Steiner point yields the trivial (identically zero) process.
	
	On the other hand, the support function satisfies
	\[
	h_{K_n}(u)=h_{K_0}(u)+\varepsilon_n h_A(u),
	\]
	so the odd component is
	\[
	C_{K_n}(u)=C_{K_0}(u)+\varepsilon_n C_A(u)=\varepsilon_n C_A(u),
	\]
	because $K_0$ is centrally symmetric and thus $C_{K_0}\equiv 0$.
	Therefore, for each direction $u$ with $C_A(u)\neq 0$,
	\[
	\Cov\big(\widetilde C_{K_n}(u),\widetilde C_{K_{n+k}}(u)\big)
	=\Cov(\varepsilon_0,\varepsilon_k)\,C_A(u)^2,
	\]
	and the location covariance $\Cov_{\mathrm{loc}}(K_0,K_k)$ directly reflects the temporal dependence of $(\varepsilon_n)$.
	This shows that Steiner-point centering can miss directional dependence that remains visible in the full odd profile.
\end{example}

\subsection{Relation to spherical harmonic representations}
\label{subsec:SH-relation}

Since the support function $h_X$ of a random convex body $X$ belongs to
$L^2(\mathbb S^{d-1},\sigma)$ under our standing moment assumptions,
it admits the classical spherical harmonic expansion
\[
h_X(u)=\sum_{\ell=0}^{\infty}\sum_{m=1}^{N(d,\ell)} a_{\ell m}(X)\,Y_{\ell m}(u),
\qquad u\in\mathbb S^{d-1},
\]
where $\{Y_{\ell m}\}$ is an orthonormal basis of spherical harmonics.
Recall that spherical harmonics satisfy the parity property
$Y_{\ell m}(-u)=(-1)^{\ell}Y_{\ell m}(u)$.

\medskip
\noindent\textbf{Even/odd decomposition.}
The loc/size decomposition introduced in Section~\ref{subsec:def-covcorr} is given by
\[
W_X(u)=\tfrac12\big(h_X(u)+h_X(-u)\big),\qquad
C_X(u)=\tfrac12\big(h_X(u)-h_X(-u)\big).
\]
In terms of spherical harmonics, this corresponds to the orthogonal
decomposition of $L^2(\mathbb S^{d-1})$ into even and odd subspaces:
$W_X$ contains only the even degrees $\ell=0,2,4,\dots$, while $C_X$
contains only the odd degrees $\ell=1,3,5,\dots$.
In particular, the degree-$1$ component represents translations of the
convex body, whereas the even component captures size and shape features
invariant under reflection.

\medskip
\noindent\textbf{Beyond a spectral interpretation.}
Although the above observation provides a harmonic interpretation of the
loc/size split, our analysis does \emph{not} proceed at the level of
spherical harmonic coefficients $\{a_{\ell m}(X)\}$.
Instead, we work directly with the support functions as random elements
in $L^2(\mathbb S^{d-1})$ and define variance, covariance, and correlation
via integrals over directions, see Definitions~\ref{def:1}.
Consequently, the dependence structure is formulated through
pointwise maximal correlations of directional evaluations
$\{h_X(u):u\in\mathbb S^{d-1}\}$ rather than through spectral
coefficients or angular power spectra.

\medskip
\noindent\textbf{Dependence and limit theorems.}
The $\rho$-mixing condition employed in this paper is defined by taking
the supremum over directions of the maximal correlation between
$h_{X_0}(u)$ and $h_{X_k}(u)$ (or their loc/size components).
This notion of dependence is geometric and direction-wise, and it does
not reduce to, nor is implied by, mixing conditions on the harmonic
coefficients.
Accordingly, the law of large numbers proved in Section~\ref{sec:limits}
is not a direct consequence of harmonic analysis or spectral theory,
but rather a probabilistic result for weakly dependent sequences of
random convex bodies within the framework of random set theory.

\medskip
\noindent\textbf{Summary.}
Spherical harmonics provide a useful language to interpret the
loc/size decomposition as an even/odd orthogonal splitting in
$L^2(\mathbb S^{d-1})$.
However, the main contributions of this paper lie in the formulation of
a correlation and mixing framework tailored to random convex bodies and
in the resulting limit theorems, which go beyond a purely spectral or
harmonic analysis approach.

\subsection{Basic properties}\label{subsec:basic-props}

\begin{proposition}[Linearity, symmetry, scaling, translation invariance]\label{prop:lin-sym-scale-shift}
	For $\star\in\{\mathrm{size},\mathrm{loc},\mathrm{tot}\}$ and any set-valued random variables $X,Y,Z$ and scalars $a,b\in\mathbb{R}$:
	\begin{enumerate}
		\item[(P1)] $\Cov_\star(X+Z,Y)=\Cov_\star(X,Y)+\Cov_\star(Z,Y)$ and $\Cov_\star(X,Y)=\Cov_\star(Y,X)$.
		\item[(P2)] $\Cov_\star(aX,bY)=ab\,\Cov_\star(X,Y)$.
		\item[(P3)] For deterministic $c\in\mathbb{R}^d$, $\Cov_\star(X+c,Y)=\Cov_\star(X,Y)$.
	\end{enumerate}
\end{proposition}
\begin{proof}
	For Minkowski addition, $h_{X+Z}=h_X+h_Z$, hence $W_{X+Z}=W_X+W_Z$ and $C_{X+Z}=C_X+C_Z$.
	After pointwise centering, $\widetilde W_{X+Z}=\widetilde W_X+\widetilde W_Z$ and
	$\widetilde C_{X+Z}=\widetilde C_X+\widetilde C_Z$.
	Linearity and symmetry of scalar covariance and integration in $u$ yield (P1).
	Scaling follows from $h_{aX}(u)=a\,h_X(u)$ for $a\ge 0$ and the identity $h_{aX}(u)=|a|h_X(-u)$ for $a<0$,
	which implies $W_{aX}=|a|W_X$ and $C_{aX}=a\,C_X$. A change of variable $u\mapsto -u$ handles the absolute value in the size part.
	For translation, $h_{X+c}(u)=h_X(u)+\langle u,c\rangle$, so $W_{X+c}=W_X$ and $C_{X+c}=C_X+\langle u,c\rangle$;
	after centering, the deterministic term cancels and the covariances remain unchanged.
\end{proof}

\begin{proposition}[Perfect sign flip for components]\label{prop:signflip}
	We have $W_{-X}=W_X$ and $C_{-X}=-C_X$. Consequently, whenever the corresponding variances are positive,
	\[
	\Corr_{\mathrm{size}}(X,-X)=1,
	\qquad
	\Corr_{\mathrm{loc}}(X,-X)=-1.
	\]
	Moreover, if $\Var_{\mathrm{tot}}(X)>0$ then
	\begin{equation}\label{eq:signflip-tot}
		\Corr_{\mathrm{tot}}(X,-X)=\pi_{\mathrm{size}}(X)-\pi_{\mathrm{loc}}(X),
	\end{equation}
	where $\pi_{\mathrm{loc}}(X):=\Var_{\mathrm{loc}}(X)/\Var_{\mathrm{tot}}(X)$ and $\pi_{\mathrm{size}}(X):=1-\pi_{\mathrm{loc}}(X)$.
\end{proposition}
\begin{proof}
	For $-X=\{-x:x\in X\}$,
	\[
	h_{-X}(u)=\sup_{x\in X}\langle u,-x\rangle=\sup_{x\in X}\langle -u,x\rangle=h_X(-u).
	\]
	Hence $W_{-X}(u)=\tfrac12(h_X(-u)+h_X(u))=W_X(u)$ and
	$C_{-X}(u)=\tfrac12(h_X(-u)-h_X(u))=-C_X(u)$.
	After centering, $\widetilde W_{-X}=\widetilde W_X$ and $\widetilde C_{-X}=-\widetilde C_X$.
	Integrating the scalar covariance identities yields
	$\Cov_{\mathrm{size}}(X,-X)=\Var_{\mathrm{size}}(X)$ and
	$\Cov_{\mathrm{loc}}(X,-X)=-\Var_{\mathrm{loc}}(X)$, hence
	$\Corr_{\mathrm{size}}(X,-X)=1$ and $\Corr_{\mathrm{loc}}(X,-X)=-1$ when defined.
	Finally,
	\[
	\Corr_{\mathrm{tot}}(X,-X)
	=\frac{\Var_{\mathrm{size}}(X)-\Var_{\mathrm{loc}}(X)}{\Var_{\mathrm{tot}}(X)}
	=\pi_{\mathrm{size}}(X)-\pi_{\mathrm{loc}}(X).
	\]
\end{proof}

\begin{proposition}[PSD covariance matrix and Cauchy--Schwarz]\label{prop:PSD-CS}
	For a finite family $X_1,\dots,X_m$ and $\star\in\{\mathrm{size},\mathrm{loc},\mathrm{tot}\}$,
	the matrix $A=(A_{ij})$ with $A_{ij}:=\Cov_\star(X_i,X_j)$ is positive semidefinite.
	Consequently $|\Corr_\star(X,Y)|\le 1$ whenever $\Var_\star(X)\Var_\star(Y)>0$.
\end{proposition}
\begin{proof}
	Fix $\star\in\{\mathrm{size},\mathrm{loc},\mathrm{tot}\}$ and define $\Phi_i(u):=\widetilde\Phi^{(\star)}_{X_i}(u)\in L^2(\Omega)$.
	For any $a\in\mathbb{R}^m$,
	\[
	a^\top A a
	=\sum_{i,j}a_i a_j\int_{\mathbb S^{d-1}}\Cov(\Phi_i(u),\Phi_j(u))\,d\sigma(u)
	=\int_{\mathbb S^{d-1}}\Var\!\Big(\sum_{i=1}^m a_i\Phi_i(u)\Big)\,d\sigma(u)\ge 0,
	\]
	so $A\succeq 0$.
	Taking $m=2$ gives $|\Cov_\star(X,Y)|\le \sqrt{\Var_\star(X)\Var_\star(Y)}$ and hence $|\Corr_\star(X,Y)|\le 1$ when defined.
\end{proof}

\begin{proposition}[Singleton and symmetric cases]\label{prop:degenerate}
	Assume \eqref{eq:moment2}.
	\begin{enumerate}
		\item[(a)] If $X=\{\xi\}$ a.s.\ for a random vector $\xi\in\mathbb{R}^d$ with $\E\|\xi\|^2<\infty$, then
		$W_X\equiv 0$, $C_X(u)=\langle u,\xi\rangle$, and $\Var_{\mathrm{size}}(X)=0$.
		Moreover,
		\[
		\Var_{\mathrm{loc}}(X)=\int_{\mathbb S^{d-1}}\Var(\langle u,\xi\rangle)\,d\sigma(u)
		=\frac{1}{d}\,\tr(\Cov(\xi)).
		\]
		\item[(b)] If $X=\{\xi\}$ and $Y=\{\eta\}$ a.s.\ with $\E\|\xi\|^2+\E\|\eta\|^2<\infty$, then
		\[
		\Cov_{\mathrm{loc}}(X,Y)=\frac{1}{d}\,\tr(\Cov(\xi,\eta)),
		\;
		\Corr_{\mathrm{loc}}(X,Y)
		=\frac{\tr(\Cov(\xi,\eta))}{\sqrt{\tr(\Cov(\xi))\,\tr(\Cov(\eta))}}.
		\]
		In particular, when $d=1$ this reduces to the Pearson correlation.
		\item[(c)] If $X=-X$ a.s.\ (central symmetry about the origin), then $C_X\equiv 0$ and $\Var_{\mathrm{loc}}(X)=0$,
		so $\Cov_{\mathrm{loc}}(X,Y)=0$ for any $Y$.
	\end{enumerate}
\end{proposition}
\begin{proof}
	For (a), $h_X(u)=\langle u,\xi\rangle$, hence $W_X\equiv 0$ and $C_X(u)=\langle u,\xi\rangle$.
	Also $\Var(\langle u,\xi\rangle)=u^\top\Sigma u$ with $\Sigma=\Cov(\xi)$ and
	$\int_{\mathbb S^{d-1}}uu^\top\,d\sigma(u)=\frac1d I_d$, giving the trace formula.
	Part (b) is analogous using $\Cov(\langle u,\xi\rangle,\langle u,\eta\rangle)=u^\top\Cov(\xi,\eta)\,u$.
	Part (c) follows from $h_X(-u)=h_X(u)$.
\end{proof}

\begin{proposition}[Contributions and bounds]\label{prop:contrib}
	Assume $\Var_{\mathrm{tot}}(X)\Var_{\mathrm{tot}}(Y)>0$. Define
	\[
	\kappa_{\mathrm{loc}}(X,Y):=\frac{\Cov_{\mathrm{loc}}(X,Y)}{\sqrt{\Var_{\mathrm{tot}}(X)\Var_{\mathrm{tot}}(Y)}},
	\;
	\kappa_{\mathrm{size}}(X,Y):=\frac{\Cov_{\mathrm{size}}(X,Y)}{\sqrt{\Var_{\mathrm{tot}}(X)\Var_{\mathrm{tot}}(Y)}}.
	\]
	Then $\Corr_{\mathrm{tot}}(X,Y)=\kappa_{\mathrm{loc}}(X,Y)+\kappa_{\mathrm{size}}(X,Y)$ and
	\[
	|\kappa_{\mathrm{loc}}(X,Y)|\le \sqrt{\pi_{\mathrm{loc}}(X)\pi_{\mathrm{loc}}(Y)},
	\;
	|\kappa_{\mathrm{size}}(X,Y)|\le \sqrt{\pi_{\mathrm{size}}(X)\pi_{\mathrm{size}}(Y)},
	\]
	hence $|\kappa_{\mathrm{loc}}(X,Y)|+|\kappa_{\mathrm{size}}(X,Y)|\le 1$ and $|\Corr_{\mathrm{tot}}(X,Y)|\le 1$.
\end{proposition}
\begin{proof}
	By Proposition~\ref{prop:PSD-CS} applied to each component,\\
	$|\Cov_{\mathrm{loc}}(X,Y)|\le \sqrt{\Var_{\mathrm{loc}}(X)\Var_{\mathrm{loc}}(Y)}$ and similarly for size.
	Divide by $\sqrt{\Var_{\mathrm{tot}}(X)\Var_{\mathrm{tot}}(Y)}$ to obtain the two bounds.\\
	For the sum bound, apply Cauchy--Schwarz to the vectors
	$(\sqrt{\pi_{\mathrm{loc}}(X)},\sqrt{\pi_{\mathrm{size}}(X)})$ and
	$(\sqrt{\pi_{\mathrm{loc}}(Y)},\sqrt{\pi_{\mathrm{size}}(Y)})$.
\end{proof}

\section{Limit theorems for $\rho$-mixing set-valued processes}\label{sec:limits}

We now establish probabilistic limit results for weakly dependent set-valued processes.
The key input is the correlation-based dependence structure introduced in
Section~\ref{sec:New_definitions}, which allows us to control second moments of
Minkowski averages through the even--odd decomposition of support functions.

\subsection{Chebyshev inequalities}
\subsubsection{Sphere-integrated Chebyshev inequality}

\begin{definition}[Weakly stationary set-valued sequence]
	A sequence $(X_i)_{i\in\mathbb Z}$ of set-valued random variables in $\mathbb R^d$
	is called \emph{weakly stationary} if:
	\begin{enumerate}
		\item for every $u\in\mathbb S^{d-1}$, the mean $\E h_{X_i}(u)$ does not depend on $i$;
		\item for every $u,v\in\mathbb S^{d-1}$ and $k\in\mathbb Z$,
		$\Cov(h_{X_i}(u),h_{X_{i+k}}(v))$ depends only on $k$.
	\end{enumerate}
\end{definition}

\begin{theorem}[Sphere-integrated Chebyshev inequality: componentwise]
	\label{thm:ChebyshevIntegrated}
	Let $X_1,\dots,X_n$ be weakly stationary.
	For $\star\in\{\mathrm{size},\mathrm{loc}\}$ define
	\[
	S_n^{(\star)}(u):=\frac1n\sum_{k=1}^n \widetilde{\Phi}^{(\star)}_{X_k}(u).
	\]
	Then for every $\varepsilon>0$,
	\begin{equation}\label{eq:ChebCompProb}
		\PP\!\left(\|S_n^{(\star)}\|_{2,\sigma}\ge\varepsilon\right)
		\le \frac{1}{\varepsilon^2}
		\int_{\mathbb S^{d-1}}\Var\!\big(S_n^{(\star)}(u)\big)\,d\sigma(u),
	\end{equation}
	and
	\begin{equation}\label{eq:ChebCompVarBound}
		\begin{aligned}
			\int_{\mathbb S^{d-1}}\Var\!\big(S_n^{(\star)}(u)\big)\,d\sigma(u)
			&\le \frac{1}{n}\Var_\star(X_1) \\
			&\quad + \frac{2}{n}\sum_{k=1}^{n-1}\Big(1-\frac{k}{n}\Big)
			\int_{\mathbb S^{d-1}}
			\big|\Cov(\widetilde{\Phi}^{(\star)}_{X_0}(u),
			\widetilde{\Phi}^{(\star)}_{X_k}(u))\big|\,d\sigma(u).
		\end{aligned}
	\end{equation}
\end{theorem}

\begin{proof}
	\emph{Step 1.} Since $\mathbb{E}[S_n^{(\star)}(u)]=0$, apply Chebyshev's inequality:
	\[
	\mathbb{P}\!\left(\|S_n^{(\star)}\|_{2,\sigma}\ge \varepsilon\right)
	=\mathbb{P}\!\left(\|S_n^{(\star)}\|_{2,\sigma}^2\ge \varepsilon^2\right)
	\le \frac{1}{\varepsilon^2}\,\mathbb{E}\big[\|S_n^{(\star)}\|_{2,\sigma}^2\big].
	\]
	Interchanging $\mathbb{E}/\int$ (Fubini–Tonelli),
	\[
	\mathbb{E}\big[\|S_n^{(\star)}\|_{2,\sigma}^2\big]
	=\int_{\mathbb{S}^{d-1}}\mathbb{E}\!\left[S_n^{(\star)}(u)^2\right]\mathrm{d}\sigma(u)
	=\int_{\mathbb{S}^{d-1}}\operatorname{Var}\!\big(S_n^{(\star)}(u)\big)\mathrm{d}\sigma(u),
	\]
	which gives \eqref{eq:ChebCompProb}.
	
	\emph{Step 2.} For fixed $u$,
	\[\begin{aligned}
		\operatorname{Var}\!\big(S_n^{(\star)}(u)\big)
		&=\frac{1}{n^2}\sum_{i,j=1}^n \operatorname{Cov}\!\big(\widetilde{\Phi}_{X_i}^{(\star)}(u),\widetilde{\Phi}_{X_j}^{(\star)}(u)\big)\\
		&=\frac{1}{n^2}\Big(\sum_{i=1}^n \mathrm{Var}(\widetilde\Phi^{(\star)}_{X_i}(u))
		+2\!\!\sum_{1\le i<j\le n}\!\!\mathrm{Cov}(\widetilde\Phi^{(\star)}_{X_i}(u),\widetilde\Phi^{(\star)}_{X_j}(u))\Big).
	\end{aligned}\]
	Stationarity allows us to replace $\mathrm{Var}(\widetilde\Phi^{(\star)}_{X_i}(u))$ by $\mathrm{Var}(\widetilde\Phi^{(\star)}_{X_0}(u))$
	and to group by lags $k=j-i$, yielding the factor $\big(1-\tfrac{k}{n}\big)$; integrating in $u$ gives the claim.
\end{proof}

\begin{theorem}[Sphere-integrated Chebyshev inequality: total version]
	\label{thm:ChebyshevIntegratedTot}
	For every $\varepsilon>0$,
	\[
	\PP\!\left(\|S_n^{(\mathrm{tot})}\|_{2,\sigma}\ge\varepsilon\right)
	\le \frac{1}{\varepsilon^2}
	\int_{\mathbb S^{d-1}}\Var\!\big(S_n^{(\mathrm{tot})}(u)\big)\,d\sigma(u).
	\]
	Moreover,
	\begin{equation}\label{eq:TotVarDecomp}
		\int_{\mathbb S^{d-1}}\Var\!\big(S_n^{(\mathrm{tot})}(u)\big)\,d\sigma(u)
		=
		\int_{\mathbb S^{d-1}}\Var\!\big(S_n^{(\mathrm{size})}(u)\big)\,d\sigma(u)
		+
		\int_{\mathbb S^{d-1}}\Var\!\big(S_n^{(\mathrm{loc})}(u)\big)\,d\sigma(u),
	\end{equation}
	and hence
	\begin{equation}\label{eq:ChebTotVarBound}
		\int_{\mathbb S^{d-1}}\Var\!\big(S_n^{(\mathrm{tot})}(u)\big)\,d\sigma(u)
		\le \text{(size bound)}+\text{(loc bound)}.
	\end{equation}
\end{theorem}
\begin{proof}
	The probability bound \eqref{eq:ChebTotVarBound} follows exactly as in
	Theorem~\ref{thm:ChebyshevIntegrated} by applying Chebyshev's inequality
	to $\|S_n^{(\mathrm{tot})}\|_{2,\sigma}$.
	
	\medskip
	\textbf{Variance decomposition.}
	Recall that
	\[
	S_n^{(\mathrm{tot})}(u)
	=
	S_n^{(\mathrm{size})}(u)+S_n^{(\mathrm{loc})}(u).
	\]
	Hence, for each $u$,
	\[
	\begin{aligned}
		\Var\!\big(S_n^{(\mathrm{tot})}(u)\big)
		&=
		\Var\!\big(S_n^{(\mathrm{size})}(u)\big)
		+
		\Var\!\big(S_n^{(\mathrm{loc})}(u)\big) \\
		&\quad
		+
		2\,\Cov\!\big(
		S_n^{(\mathrm{size})}(u),
		S_n^{(\mathrm{loc})}(u)
		\big).
	\end{aligned}
	\]
	
	Integrating over $\mathbb S^{d-1}$,
	the cross-term vanishes due to even--odd orthogonality:
	\[
	\int_{\mathbb S^{d-1}}
	\E\!\big[
	\widetilde W_{X_i}(u)\,\widetilde C_{X_j}(u)
	\big]\,d\sigma(u)
	=0,
	\qquad \forall\, i,j.
	\]
	Therefore,
	\[
	\int_{\mathbb S^{d-1}}
	\Var\!\big(S_n^{(\mathrm{tot})}(u)\big)\,d\sigma(u)
	=
	\int_{\mathbb S^{d-1}}
	\Var\!\big(S_n^{(\mathrm{size})}(u)\big)\,d\sigma(u)
	+
	\int_{\mathbb S^{d-1}}
	\Var\!\big(S_n^{(\mathrm{loc})}(u)\big)\,d\sigma(u),
	\]
	which proves \eqref{eq:TotVarDecomp}.
	
	\medskip
	\textbf{Final bound.}
	Applying the componentwise bounds \eqref{eq:ChebCompVarBound}
	to the size and location parts separately yields
	\eqref{eq:ChebTotVarBound}.
\end{proof}

\begin{remark}
	The ``total'' inequality can equivalently be written in terms of the centered
	support functions $\widetilde h_{X_i}$.
	However, the decomposition \eqref{eq:TotVarDecomp} highlights a structural fact
	specific to random sets: size and location fluctuations are orthogonal in
	$L^2(\mathbb S^{d-1},\sigma)$.
	This separation will be crucial for defining compatible mixing coefficients
	and for deriving limit theorems under weak dependence.
\end{remark}

\begin{remark}[Note on weak stationarity]
	The weighted factor $(1-k/n)$ in \eqref{eq:ChebCompVarBound}–\eqref{eq:ChebTotVarBound} is the standard form when $(X_k)$ is weakly stationary,
	in which case the covariance integrals depend only on the lag $k$. 
	If not stationary, one may replace the integrals $\int_{\Sph}|\Cov(\cdot,\cdot)|$ by $\sup_i$ over time indices to obtain analogous bounds.
\end{remark}
	
\subsubsection{Compatible $\rho$-mixing coefficient}

\begin{definition}[Maximal correlation (HGR) \cite{bradley2005basic}]
	For two random variables $U,V$ (not necessarily in $L^2(P)$, but all transformations considered will have finite second moments), let $P_U, P_V$ denote their marginal laws. Define
	\[
	\begin{aligned}
		\rho(U,V)
		:= & \sup\Big\{ \Corr\big(f(U),\,g(V)\big):
		\ f\in L^2(P_U),\ g\in L^2(P_V),\\
		& \E[f(U)]=\E[g(V)]=0,\ \Var(f(U))=\Var(g(V))=1 \Big\}\in[0,1].
	\end{aligned}
	\]
	Equivalently,
	\[
	\rho(U,V)=\sup\Big\{ \E\big[f(U)g(V)\big]:
	\ f\in L^2_0(P_U),\ g\in L^2_0(P_V),\ \|f(U)\|_2=\|g(V)\|_2=1 \Big\}.
	\]
\end{definition}

From the definition, it is immediate that for all $U,V\in L^2(P)$ one has the standard inequality
\[
|\Cov(U,V)|\ \le\ \rho(U,V)\,\sqrt{\Var(U)\,\Var(V)}.
\]

\begin{definition}[Compatible $\rho$-mixing]\label{def:compat_rho}
	For a set-valued sequence $(X_i)$, define
	\[
	\rho_{\max}(k):=\max_{\star\in\{\mathrm{size},\mathrm{loc},\mathrm{tot}\}}\ 
	\sup_{u\in\mathbb S^{d-1}}
	\rho\!\Big(\widetilde\Phi^{(\star)}_{X_0}(u),\ \widetilde\Phi^{(\star)}_{X_k}(u)\Big),\qquad k\ge 1.
	\]
	For two sequences $(X_i)$ and $(Y_i)$, define
	\[
	\rho^{XY}_{\max}(k):=\max_{\star\in\{\mathrm{size},\mathrm{loc},\mathrm{tot}\}}
	\ \sup_{u\in\mathbb S^{d-1}}
	\rho\!\Big(\widetilde\Phi^{(\star)}_{X_0}(u),\ \widetilde\Psi^{(\star)}_{Y_k}(u)\Big).
	\]
\end{definition}

\begin{remark}[Compatibility with classical $\rho$-mixing and sharpness]\ \\
	\label{rem:rho-comparison}
	
	\textbf{(a) Domination by classical $\rho$-mixing.} 
	For $\star=\mathrm{tot}$ we have $\Phi^{(\mathrm{tot})}_X=h_X$, so $\rho_{\max}(k)$ contains the factor $\sup_{u\in\mathbb{S}^{d-1}}\rho\!\big(\widetilde{h}_{X_0}(u),\widetilde{h}_{X_k}(u)\big)$. Since we take the supremum over a subclass of $L^2$ functions (the even/odd/total support projections indexed by $u\in\mathbb{S}^{d-1}$), it follows that
	\[
	\rho_{\max}(k)\le\rho_{\mathrm{classic}}(k),
	\]
	where $\rho_{\mathrm{classic}}(k)$ is the classical HGR maximal correlation coefficient $\rho\!\big(\sigma(X_0),\sigma(X_k)\big)$ between the $\sigma$-algebras $\sigma(X_0)$ and $\sigma(X_k)$. Hence, if $\sum_k\rho_{\mathrm{classic}}(k)<\infty$, then also $\sum_k\rho_{\max}(k)<\infty$.
	
	\textbf{(b) Proof of the inequality.} 
	Let $\mathcal{F}_0=\sigma(X_0)$ and $\mathcal{F}_k=\sigma(X_k)$. For any $u\in\mathbb{S}^{d-1}$ and $\star\in\{\mathrm{size},\mathrm{loc},\mathrm{tot}\}$, the random variables $\widetilde{\Phi}_{X_0}^{(\star)}(u)$ and $\widetilde{\Phi}_{X_k}^{(\star)}(u)$ are $\mathcal{F}_0$- and $\mathcal{F}_k$-measurable, respectively. By definition of HGR maximal correlation,
	\[
	\rho\!\big(\widetilde{\Phi}_{X_0}^{(\star)}(u),\widetilde{\Phi}_{X_k}^{(\star)}(u)\big)
	\le\rho(\mathcal{F}_0,\mathcal{F}_k)=\rho_{\mathrm{classic}}(k).
	\]
	Taking the supremum over $u\in\mathbb{S}^{d-1}$ and $\max$ over $\star\in\{\mathrm{size},\mathrm{loc},\mathrm{tot}\}$ yields $\rho_{\max}(k)\le\rho_{\mathrm{classic}}(k)$.
	
	\textbf{(c) Strictness of the inequality.} 
	The coefficient $\rho_{\max}$ is dominated by the classical HGR coefficient
	$\rho_{\mathrm{classic}}$, hence summability of $\rho_{\mathrm{classic}}$
	implies summability of $\rho_{\max}$.
	The inequality can be strict, see Example \ref{ex:strictness}.
	
	\textbf{(d) Practical implication.} 
	The coefficient $\rho_{\max}$ is geometrically meaningful: it captures dependence in the support function projections (which fully determine the set), while ignoring extraneous information in $\sigma(X_i)$ that does not affect the set's geometry. This makes $\rho_{\max}$ a natural and often more tractable mixing coefficient for set-valued processes.
\end{remark}
\begin{example}\label{ex:strictness}
	Let $X_i$ be random discs of radius $R_i$ centered at the origin, where $(R_i)$ is a stationary AR(1) process: $R_i=\varphi R_{i-1}+\varepsilon_i$ with $|\varphi|<1$ and $\varepsilon_i$ i.i.d.\ mean-zero. Then:
	\begin{itemize}
		\item The location component vanishes: $C_{X_i}(u)\equiv 0$ for all $u$.
		\item The size component $W_{X_i}(u)=R_i$ is constant in $u$.
		\item Hence $\rho\!\big(\widetilde{W}_{X_0}(u),\widetilde{W}_{X_k}(u)\big)=|\mathrm{Corr}(R_0,R_k)|=|\varphi|^k$ for all $u$.
		\item Thus $\rho_{\max}(k)=|\varphi|^k$.
	\end{itemize}
	However, the $\sigma$-algebra $\sigma(X_0)$ contains all information about $R_0$, and similarly for $\sigma(X_k)$. If we introduce additional randomness in $X_0$ that is independent of $R_0$ (e.g., a random marking or label), then $\sigma(X_0)$ becomes strictly richer while $\rho_{\max}(k)$ remains unchanged. In such cases, $\rho_{\mathrm{classic}}(k)$ can be strictly larger than $\rho_{\max}(k)$.
\end{example}

\subsection{Law of large numbers under the $L^2(\sigma)$ support norm}\label{subsec:LLN}

Throughout this subsection, we work on $(\Omega,\mathcal F,\PP)$ and assume that
$(X_i)_{i\in\mathbb Z}$ is weakly stationary and satisfies
\[
\sum_{k=1}^{\infty}\rho_{\max}(k)<\infty.
\]
We write $\mathcal H:=L^2(\mathbb S^{d-1},\sigma)$ and $\|\cdot\|_{\mathcal H}=\|\cdot\|_{2,\sigma}$.

\begin{lemma}[Covariance bound via $\rho_{\max}$]\label{lem:cov-rho}
	For every $k\ge1$ and $\star\in\{\mathrm{size},\mathrm{loc}\}$,
	\[
	\Gamma_\star(k)\ \le\ \rho_{\max}(k)\,\Var_\star(X_0).
	\]
\end{lemma}
\begin{proof}
	By the definition of maximal correlation, for each $u\in\mathbb S^{d-1}$,
	\[
	\big|\Cov(U,V)\big|\le \rho(U,V)\sqrt{\Var(U)\,\Var(V)}.
	\]
	Apply this with $U=\Phi^{(\star)}_{X_0}(u)$ and $V=\Phi^{(\star)}_{X_k}(u)$, take the supremum in $u$
	on $\rho$, and integrate in $u$ to obtain
	\[
	\Gamma_\star(k)
	\le
	\rho_{\max}(k)\int_{\mathbb S^{d-1}}
	\Var\!\big(\Phi^{(\star)}_{X_0}(u)\big)\,d\sigma(u)
	=
	\rho_{\max}(k)\,\Var_\star(X_0).
	\]
\end{proof}

Combining Theorem~\ref{thm:ChebyshevIntegrated} with Lemma~\ref{lem:cov-rho} yields
\begin{equation}\label{eq:var-int-bound}
	\int_{\mathbb S^{d-1}}\Var\!\left(S_n^{(\star)}(u)\right)\,d\sigma(u)
	\ \le\ \frac{\Var_\star(X_0)}{n}\left(1+2\sum_{k=1}^{n-1}\rho_{\max}(k)\right),
	\qquad \star\in\{\mathrm{size},\mathrm{loc}\}.
\end{equation}
In particular, by stationarity we may set
\[
K_\rho:=\Var_\star(X_0)\Big(1+2\sum_{k\ge1}\rho_{\max}(k)\Big),
\]
so that for all $m\ge0$ and $\ell\ge1$,
\begin{equation}\label{eq:block-second-moment}
	\mathbb E\Big\|\sum_{i=m+1}^{m+\ell}\!\!\Phi^{(\star)}_{X_i}\Big\|_{2,\sigma}^2\ \le\ K_\rho\,\ell.
\end{equation}

\subsubsection{WLLN under the $L^2(\sigma)$ norm}

\begin{theorem}[WLLN]\label{thm:WLLN}
	Assume $\int_{\mathbb S^{d-1}}\mathbb E\,h_{X_0}(u)^2\,d\sigma(u)<\infty$
	and $\sum_{k\ge1}\rho_{\max}(k)<\infty$.
	Then, for $\star\in\{\mathrm{size},\mathrm{loc}\}$,
	\[
	\|S_n^{(\star)}\|_{2,\sigma}\ \xrightarrow{\ \PP\ }\ 0.
	\]
\end{theorem}
\begin{proof}
	From \eqref{eq:var-int-bound} and Theorem~\ref{thm:ChebyshevIntegrated},
	\[
	\PP\!\left(\|S_n^{(\star)}\|_{2,\sigma}\ge\varepsilon\right)
	\le \frac{K_\rho}{\varepsilon^2\,n}
	\xrightarrow[n\to\infty]{} 0,
	\qquad \star\in\{\mathrm{size},\mathrm{loc}\}.
	\]
\end{proof}

\subsubsection{SLLN under the $L^2(\sigma)$ norm}

\begin{theorem}[SLLN]\label{thm:SLLN}
	Assume $\sum_{k\ge1}\rho_{\max}(k)<\infty$ and
	$\int_{\mathbb S^{d-1}}\mathbb E\,h_{X_0}(u)^2\,d\sigma(u)<\infty$.
	Then, for $\star\in\{\mathrm{size},\mathrm{loc}\}$,
	\[
	\|S_n^{(\star)}\|_{2,\sigma}\ \xrightarrow{\ \mathrm{a.s.}\ }\ 0.
	\]
\end{theorem}
\begin{proof}
	Let $Z_i:=\Phi^{(\star)}_{X_i}\in\mathcal H:=L^2(\mathbb S^{d-1},\sigma)$ with $\E Z_i=0$ and
	$\|\cdot\|_{\mathcal H}=\|\cdot\|_{2,\sigma}$. Denote partial sums $T_n:=\sum_{i=1}^n Z_i$
	and averages $S_n:=T_n/n$ (i.e.\ $S_n^{(\star)}$).
	
	\medskip
	\noindent\emph{Step 1 (Convergence along a dyadic subsequence).}
	Take $n=2^r$. From \eqref{eq:block-second-moment} with $m=0$ and $\ell=2^r$,
	\[
	\E\|S_{2^r}\|_{\mathcal H}^2
	=
	\frac{1}{2^{2r}}\E\|T_{2^r}\|_{\mathcal H}^2
	\le \frac{K_\rho}{2^r}.
	\]
	By Markov,
	\[
	\PP(\|S_{2^r}\|_{\mathcal H}>\varepsilon)\le \frac{K_\rho}{\varepsilon^2\,2^r},
	\]
	so $\sum_{r\ge1}\PP(\|S_{2^r}\|_{\mathcal H}>\varepsilon)<\infty$. By Borel--Cantelli,
	$\|S_{2^r}\|_{\mathcal H}\to 0$ a.s.
	
	\medskip
	\noindent\emph{Step 2 (Maximal bound within a dyadic block).}
	Fix $r\ge1$, let $L:=2^r$ and $m:=2^r$. For the within-block sums
	$U_\ell:=\sum_{i=1}^\ell Z_{m+i}$, $1\le\ell\le L$, we show that
	\begin{equation}\label{eq:block-max-app}
		\sum_{r=1}^\infty
		\PP\!\left(\max_{1\le\ell\le 2^r}\|U_\ell\|_{\mathcal H}>\varepsilon\,2^r\right)
		<\infty,
		\quad\Rightarrow\quad
		\frac{1}{2^r}\max_{1\le\ell\le 2^r}\|U_\ell\|_{\mathcal H}\xrightarrow{\mathrm{a.s.}}0.
	\end{equation}
	
	Use dyadic chaining. For each $j=0,1,\dots,r$, define grid points
	$t_{j,k}:=k\,2^{r-j}$, $k=0,1,\dots,2^j$. For $1\le \ell\le L$ and
	$p_j(\ell):=t_{j,\lfloor \ell/2^{r-j}\rfloor}$, we have $p_0(\ell)=0$, $p_r(\ell)=\ell$ and
	\[
	U_\ell=\sum_{j=1}^{r}\big(U_{p_j(\ell)}-U_{p_{j-1}(\ell)}\big).
	\]
	Hence
	\[
	\|U_\ell\|_{\mathcal H}\le \sum_{j=1}^{r}
	\max_{1\le k\le 2^j}\big\|U_{t_{j,k}}-U_{t_{j,k-1}}\big\|_{\mathcal H}.
	\]
	Therefore,
	\[
	\max_{1\le\ell\le L}\|U_\ell\|_{\mathcal H}
	\le \sum_{j=0}^{r} M_j,
	\qquad
	M_j:=\max_{1\le k\le 2^j}\big\|U_{t_{j,k}}-U_{t_{j,k-1}}\big\|_{\mathcal H}.
	\]
	Let $t:=\varepsilon L/(2(r+1))=\varepsilon 2^{r-1}/(r+1)$. By a union bound,
	\[
	\PP\!\left(\max_{\ell\le L}\|U_\ell\|_{\mathcal H}>\varepsilon L\right)
	\le \sum_{j=0}^{r}\PP(M_j>t).
	\]
	
	For each $j$, there are $2^{j}$ disjoint blocks $D_1,\ldots,D_{2^j}$ of consecutive indices,
	each of length $2^{r-j}$. Let $S_k^{(j)}:=\sum_{i\in D_k} Z_{m+i}$ for $k=1,\ldots,2^j$.
	Then $M_j=\max_{1\le k\le 2^j}\|S_k^{(j)}\|_{\mathcal H}$.
	By stationarity, all $S_k^{(j)}$ have the same distribution. Hence, by Markov and
	\eqref{eq:block-second-moment},
	\[
	\begin{aligned}
		\PP(M_j>t)
		&\le \sum_{k=1}^{2^j}\PP\!\left(\|S_k^{(j)}\|_{\mathcal H}>t\right)
		\le 2^j\,\frac{\E\|S_1^{(j)}\|_{\mathcal H}^2}{t^2}
		\le 2^j\,\frac{K_\rho\,2^{r-j}}{t^2}
		= \frac{K_\rho\,2^{r}}{t^2}.
	\end{aligned}
	\]
	Substituting $t=\varepsilon 2^{r-1}/(r+1)$ yields
	\[
	\PP(M_j>t)\le \frac{4K_\rho (r+1)^2}{\varepsilon^2\,2^{r}}.
	\]
	Since $\sum_{r\ge1}(r+1)^3/2^r<\infty$, this implies \eqref{eq:block-max-app}.
	
	\medskip
	\noindent\emph{Step 3 (Filling in between dyadics).}
	For $2^r<n\le 2^{r+1}$, write
	\[
	S_n
	= \frac{1}{n}\sum_{i=1}^{2^r} Z_i + \frac{1}{n}\sum_{i=2^r+1}^{n} Z_i
	= \frac{2^r}{n}\,S_{2^r} + \frac{1}{n}\,U_{n-2^r}.
	\]
	Hence
	\[
	\|S_n\|_{\mathcal H}
	\le \|S_{2^r}\|_{\mathcal H}
	+ \frac{1}{2^r}\max_{1\le\ell\le 2^r}\|U_\ell\|_{\mathcal H}.
	\]
	By Steps 1 and 2, both terms converge to $0$ a.s.\ as $r\to\infty$, hence
	$\|S_n\|_{\mathcal H}\to0$ a.s.\ as $n\to\infty$.
\end{proof}

\begin{remark}
	The argument uses only second moments and the summability condition
	$\sum_k\rho_{\max}(k)<\infty$, through the uniform block second-moment bound
	\eqref{eq:block-second-moment}. For stronger maximal inequalities under mixing
	conditions, see Merlev\`ede et al.~\cite{merlevede2009bernstein} and
	Rio~\cite{rio2017asymptotic}.
\end{remark}

\subsubsection{Geometric consequences and the total norm}

Refer to Aumann expectation and support/Mosco/Hausdorff convergence in
\cite{Molchanov2005}. We record the following consequences.

\begin{corollary}
	The Minkowski average $\overline{X}_n:=\tfrac1n(X_1+\cdots+X_n)$ converges in support
	to the Aumann expectation $\E[X_0]$, characterized by the pair
	$\big(\E W_{X_0},\,\E C_{X_0}\big)$.
\end{corollary}

\begin{corollary}
	If $X_i$ and $\E[X_0]$ are convex, closed, and uniformly bounded, then the above support
	convergence implies Mosco convergence; moreover, by the uniform Lipschitz property of
	$h_X$ in $u$, one also has Hausdorff convergence.
\end{corollary}

\paragraph{Total norm}
Let $T_X:=\widetilde W_X+\widetilde C_X$. Then
\[
\Big\|\frac{1}{n}\sum_{i=1}^n T_{X_i}\Big\|_{2,\sigma}
\ \le\ \|S_n^{(\mathrm{size})}\|_{2,\sigma}+\|S_n^{(\mathrm{loc})}\|_{2,\sigma}.
\]
Hence the WLLN/SLLN also hold for the total norm, which ensures consistency of the
estimators $\Corr_{\mathrm{tot}}$, $\kappa$, and $\pi$.


\subsection{Stronger limit theorems under $\rho_{\max}$--mixing}\label{subsec:stronger}

\subsubsection{Setup, notation, and assumptions}\label{subsubsec:setup}

Let $\mathcal H:=L^2(\mathbb S^{d-1},\sigma)$ with inner product
$\langle f,g\rangle=\int_{\mathbb S^{d-1}} f(u)g(u)\,d\sigma(u)$ and norm
$\|\cdot\|_{\mathcal H}=\|\cdot\|_{2,\sigma}$. For a weakly stationary set--valued sequence
$(X_i)_{i\in\mathbb Z}$, define
\[
Z_i^{(\mathrm{size})}(u):=W_{X_i}(u)-\E W_{X_0}(u),\qquad
Z_i^{(\mathrm{loc})}(u):=C_{X_i}(u)-\E C_{X_0}(u),
\]
and $Z_i^{(\mathrm{tot})}:=Z_i^{(\mathrm{size})}+Z_i^{(\mathrm{loc})}$.
Put $T_n^{(\star)}:=\sum_{i=1}^n Z_i^{(\star)}$ and $S_n^{(\star)}:=n^{-1}T_n^{(\star)}$ for
$\star\in\{\mathrm{size},\mathrm{loc},\mathrm{tot}\}$.

\paragraph{Maximal compatible correlation.}
For $k\ge1$, set
\[
\rho_{\max}(k):=\max_{\star\in\{\mathrm{size},\mathrm{loc},\mathrm{tot}\}}
\ \sup_{u\in\mathbb S^{d-1}}
\rho\!\big(\widetilde\Phi_{X_0}^{(\star)}(u), \widetilde\Phi_{X_k}^{(\star)}(u)\big),
\qquad \widetilde\Phi=\Phi-\E\Phi.
\]

\paragraph{Long--run covariance operators.}
For $\star\in\{\mathrm{size},\mathrm{loc}\}$ and $k\in\mathbb Z$ define kernels
\begin{equation}\label{eq:lrv-kernel}
	K_k^{(\star)}(u,v):=\Cov\big(Z_0^{(\star)}(u),Z_k^{(\star)}(v)\big),
\end{equation}
and integral operators
\[
(C_k^{(\star)}f)(u):=\int_{\mathbb S^{d-1}} K_k^{(\star)}(u,v)\,f(v)\,d\sigma(v),
\qquad
\Sigma^{(\star)}:=\sum_{k\in\mathbb Z} C_k^{(\star)}.
\]

\paragraph{Standing assumptions.}
\begin{itemize}
	\item[(M$2+$)] $\displaystyle \int_{\mathbb S^{d-1}} \E\,|h_{X_0}(u)|^{2+\delta}\,d\sigma(u)<\infty$ for some $\delta>0$.
	\item[(Mix)] $\displaystyle \sum_{k\ge1}\rho_{\max}(k)^{\delta/(2+\delta)}<\infty$.
\end{itemize}

\begin{remark}
	Under \textnormal{(M$2+$)} and \textnormal{(Mix)}, Lemma~\ref{lem:lrv} ensures that
	$\Sigma^{(\star)}=\sum_{k\in\mathbb Z}C_k^{(\star)}$ exists as a trace--class operator
	for $\star\in\{\mathrm{size},\mathrm{loc}\}$. We refer to this conclusion as
	\textnormal{(LRV)} (long--run variance exists). Thus \textnormal{(LRV)} is implied by
	\textnormal{(M$2+$)} $+$ \textnormal{(Mix)} and is invoked later only for readability.
\end{remark}

\subsubsection{Even/odd parity}\label{subsubsec:parity}
Let $\mathcal{H}_{\mathrm{even}}:=\{f:f(u)=f(-u)\}$ and
$\mathcal{H}_{\mathrm{odd}}:=\{f:f(u)=-f(-u)\}$. Then
$\mathcal{H}=\mathcal{H}_{\mathrm{even}}\oplus \mathcal{H}_{\mathrm{odd}}$ orthogonally,
$Z_i^{(\mathrm{size})}\in \mathcal{H}_{\mathrm{even}}$,
$Z_i^{(\mathrm{loc})}\in \mathcal{H}_{\mathrm{odd}}$, and cross--covariances vanish upon
sphere integration. Consequently,
$\Sigma^{(\mathrm{tot})}=\Sigma^{(\mathrm{size})}\oplus \Sigma^{(\mathrm{loc})}$.

\subsubsection{Auxiliary lemmas}\label{subsubsec:lemmas}

\begin{lemma}[Maximal correlation inequality]\label{lem:maxcorr}
	Let $\mathcal F,\mathcal G$ be $\sigma$--fields and
	\(
	\rho(\mathcal F,\mathcal G):=\sup\{|\Corr(U,V)|:\ U\in L^2(\mathcal F),\ V\in L^2(\mathcal G),\ \E U=\E V=0\}.
	\)
	Then for any $U\in L^2(\mathcal F)$ and $V\in L^2(\mathcal G)$,
	\[
	|\Cov(U,V)|\ \le\ \rho(\mathcal F,\mathcal G)\,\sqrt{\Var(U)\Var(V)}.
	\]
\end{lemma}
\begin{proof}
	Without loss of generality assume $\E U=\E V=0$ (otherwise replace by centered versions).
	Then $|\Corr(U,V)|\le \rho(\mathcal F,\mathcal G)$ by definition, i.e.
	\[
	\frac{|\Cov(U,V)|}{\sqrt{\Var(U)\Var(V)}}\le \rho(\mathcal F,\mathcal G),
	\]
	which yields the claim.
\end{proof}

\begin{lemma}[Parity orthogonality]\label{lem:parity}
	$Z_i^{(\mathrm{size})}\in \mathcal{H}_{\mathrm{even}}$ and
	$Z_i^{(\mathrm{loc})}\in \mathcal{H}_{\mathrm{odd}}$. Consequently,
	\[
	\int_{\mathbb S^{d-1}}\E\big[Z_i^{(\mathrm{size})}(u)\,Z_j^{(\mathrm{loc})}(u)\big]\,d\sigma(u)=0,
	\]
	and $C_k^{(\mathrm{size})}$ acts on $\mathcal{H}_{\mathrm{even}}$, while
	$C_k^{(\mathrm{loc})}$ acts on $\mathcal{H}_{\mathrm{odd}}$ for all $k$.
\end{lemma}
\begin{proof}
	$W_X(u)=\tfrac12(h_X(u)+h_X(-u))$ is even in $u$, while
	$C_X(u)=\tfrac12(h_X(u)-h_X(-u))$ is odd. Centering preserves parity, hence
	$Z^{(\mathrm{size})}$ is even and $Z^{(\mathrm{loc})}$ is odd.
	The product of an even and an odd function is odd, so its $\sigma$--integral over the sphere is $0$.
	This gives the stated orthogonality. The invariance of $\mathcal H_{\mathrm{even}}$ and
	$\mathcal H_{\mathrm{odd}}$ under $C_k^{(\mathrm{size})}$ and $C_k^{(\mathrm{loc})}$ follows from
	the parity of the kernels.
\end{proof}

\paragraph{Hilbert--Schmidt operators on $\mathcal H=L^2(\mathbb S^{d-1},\sigma)$.}
A bounded linear operator $T:\mathcal H\to \mathcal H$ is \emph{Hilbert--Schmidt (HS)} if
\[
\|T\|_{\mathrm{HS}}^2:=\sum_{j=1}^\infty \|T e_j\|_{\mathcal H}^2<\infty,
\]
for (equivalently, for any) orthonormal basis $\{e_j\}_{j\ge1}$ of $\mathcal H$.
If $T$ is an integral operator with kernel
$K\in L^2(\mathbb S^{d-1}\times \mathbb S^{d-1},\sigma\otimes\sigma)$,
\[
(Tf)(u)=\int_{\mathbb S^{d-1}} K(u,v)\,f(v)\,d\sigma(v),
\]
then $T$ is HS and
\[
\|T\|_{\mathrm{HS}}^2=
\iint_{\mathbb S^{d-1}\times \mathbb S^{d-1}} |K(u,v)|^2\,d\sigma(u)\,d\sigma(v).
\]
(See, e.g., \cite{conway2019course}.)
In particular, for $C_k^{(\star)}$ with kernel
$K_k^{(\star)}(u,v):=\Cov\big(Z_0^{(\star)}(u),Z_k^{(\star)}(v)\big)$,
we have $\|C_k^{(\star)}\|_{\mathrm{HS}}^2=\iint |K_k^{(\star)}(u,v)|^2\,d\sigma(u)\,d\sigma(v)$
whenever $K_k^{(\star)}\in L^2(\sigma\otimes\sigma)$.

\begin{lemma}[HS control and existence of $\Sigma^{(\star)}$]\label{lem:lrv}
	Under \textnormal{(M$2+$)} and \textnormal{(Mix)}, we have
	$\sum_{k\in\mathbb Z}\|C_k^{(\star)}\|_{\mathrm{HS}}<\infty$, hence $\Sigma^{(\star)}$ exists
	and is trace--class. In particular, the stronger condition
	$\sum_k\rho_{\max}(k)<\infty$ also suffices.
\end{lemma}
\begin{proof}
	By Lemma~\ref{lem:maxcorr}, for all $u,v$ and $k$,
	\[
	|K_k^{(\star)}(u,v)|
	\le
	\rho_{\max}(k)\sqrt{\Var(Z_0^{(\star)}(u))\,\Var(Z_0^{(\star)}(v))}.
	\]
	Therefore,
	\[
	\|C_k^{(\star)}\|_{\mathrm{HS}}^2
	=
	\iint K_k^{(\star)}(u,v)^2\,d\sigma(u)\,d\sigma(v)
	\le
	\rho_{\max}(k)^2
	\Big(\int \Var(Z_0^{(\star)}(u))\,d\sigma(u)\Big)^2.
	\]
	Hence $\sum_k\|C_k^{(\star)}\|_{\mathrm{HS}}<\infty$ if $\sum_k\rho_{\max}(k)<\infty$, since
	\[
	\int \Var(Z_0^{(\star)}(u))\,d\sigma(u)
	\le
	\int \E|Z_0^{(\star)}(u)|^2\,d\sigma(u)<\infty
	\]
	by \textnormal{(M$2+$)}.
	
	With only \textnormal{(Mix)}, apply a Davydov/Rio-type bound for $\rho$--mixing (see, e.g., \cite{rio2017asymptotic}):
	for some constant $C_\delta$, for all $u,v$,
	\[
	|K_k^{(\star)}(u,v)|
	\le
	C_\delta\,\rho_{\max}(k)^{\delta/(2+\delta)}
	\big(\E|Z_0^{(\star)}(u)|^{2+\delta}\big)^{1/(2+\delta)}
	\big(\E|Z_0^{(\star)}(v)|^{2+\delta}\big)^{1/(2+\delta)}.
	\]
	Square and integrate in $(u,v)$; by H\"older and \textnormal{(M$2+$)},
	\[
	\|C_k^{(\star)}\|_{\mathrm{HS}}
	\le
	C_\delta\,\rho_{\max}(k)^{\delta/(2+\delta)}
	\Big(\int \E|Z_0^{(\star)}(u)|^{2+\delta}\,d\sigma(u)\Big)^{2/(2+\delta)}.
	\]
	Summing in $k$ and using \textnormal{(Mix)} yields $\sum_{k\in\mathbb Z}\|C_k^{(\star)}\|_{\mathrm{HS}}<\infty$.
	Therefore $\Sigma^{(\star)}=\sum_{k\in\mathbb Z}C_k^{(\star)}$ exists and is trace--class.
\end{proof}

\begin{lemma}[Second--moment bound for increments]\label{lem:increments}
	There exists $K_\rho<\infty$ such that, for all $m,\ell\ge1$ and
	$\star\in\{\mathrm{size},\mathrm{loc}\}$,
	\[
	\E\Big\|\sum_{i=m+1}^{m+\ell}Z_i^{(\star)}\Big\|_{\mathcal H}^2\ \le\ K_\rho\,\ell.
	\]
\end{lemma}
\begin{proof}
	Write
	\[
	\E\Big\|\sum_{i=m+1}^{m+\ell}Z_i^{(\star)}\Big\|_{\mathcal H}^2
	=
	\sum_{r,s=1}^\ell\int
	\Cov\big(Z_{m+r}^{(\star)}(u),Z_{m+s}^{(\star)}(u)\big)\,d\sigma(u)
	=
	\sum_{h=-(\ell-1)}^{\ell-1}(\ell-|h|)\,a_h,
	\]
	where $a_h:=\int \Cov\big(Z_0^{(\star)}(u),Z_h^{(\star)}(u)\big)\,d\sigma(u)$.
	By Cauchy--Schwarz in $\mathcal H$, $|a_h|\le \|C_h^{(\star)}\|_{\mathrm{HS}}$.
	Hence
	\[
	\E\Big\|\sum_{i=m+1}^{m+\ell}Z_i^{(\star)}\Big\|_{\mathcal H}^2
	\le
	\ell\sum_{h\in\mathbb Z}\|C_h^{(\star)}\|_{\mathrm{HS}}
	=:K_\rho\,\ell,
	\]
	which is finite by Lemma~\ref{lem:lrv}.
\end{proof}

\subsubsection{Hilbert--space CLT}\label{subsubsec:HCLT}

\begin{theorem}[CLT in ${\mathcal H}$]\label{thm:HCLT}
	Under \textnormal{(M$2+$)}, \textnormal{(Mix)}, and \textnormal{(LRV)}, we have
	\[
	\sqrt{n}\,S_n^{(\mathrm{size})}\Rightarrow \mathcal N_{\mathcal H}\!\big(0,\Sigma^{(\mathrm{size})}\big),\qquad
	\sqrt{n}\,S_n^{(\mathrm{loc})}\Rightarrow \mathcal N_{\mathcal H}\!\big(0,\Sigma^{(\mathrm{loc})}\big).
	\]
	Consequently,
	$\sqrt{n}\,S_n^{(\mathrm{tot})}\Rightarrow G^{(\mathrm{size})}+G^{(\mathrm{loc})}$ with
	block--diagonal covariance $\Sigma^{(\mathrm{size})}\oplus\Sigma^{(\mathrm{loc})}$ on
	${\mathcal H}_{\mathrm{even}}\oplus {\mathcal H}_{\mathrm{odd}}$.
\end{theorem}
\begin{proof}
	\emph{Step 1 (finite--dimensional convergence).}
	Fix $h_1,\ldots,h_M\in {\mathcal H}$ and consider
	\[
	Y_n:=\frac{1}{\sqrt n}\sum_{i=1}^n\big(\langle Z_i^{(\star)},h_1\rangle,\ldots,\langle Z_i^{(\star)},h_M\rangle\big)\in\mathbb R^M.
	\]
	Each coordinate is a centered, weakly stationary real process, and its mixing is controlled by $\rho_{\max}(k)$
	by linearity of $\langle\cdot,h\rangle$. Using the big--block/small--block method with block length $b_n\to\infty$,
	gap $s_n\to\infty$, and $s_n/b_n\to0$, assumption \textnormal{(Mix)} renders separated big blocks asymptotically independent.
	The contribution of gaps is negligible in $L^2$, and the Lindeberg condition follows from \textnormal{(M$2+$)}.
	Hence $Y_n\Rightarrow \mathcal N_M(0,\Gamma)$ with
	\[
	\Gamma_{j\ell}
	=
	\sum_{k\in\mathbb Z}\Cov\big(\langle Z_0^{(\star)},h_j\rangle,\langle Z_k^{(\star)},h_\ell\rangle\big)
	=
	\langle \Sigma^{(\star)}h_j,h_\ell\rangle.
	\]
	
	\emph{Step 2 (tightness in ${\mathcal H}$).}
	Let $\{e_r\}_{r\ge1}$ be an orthonormal basis of $\mathcal H$ and $P_m$ the projection onto
	$\mathrm{span}\{e_1,\dots,e_m\}$. Decompose
	\[
	\frac1{\sqrt n}\sum_{i=1}^n Z_i^{(\star)}
	=
	\frac1{\sqrt n}\sum_{i=1}^n P_m Z_i^{(\star)}
	+
	\frac1{\sqrt n}\sum_{i=1}^n (I-P_m)Z_i^{(\star)}.
	\]
	The first term is finite--dimensional and tight by Step~1. For the tail,
	\[
	\E\Big\|\frac1{\sqrt n}\sum_{i=1}^n (I-P_m)Z_i^{(\star)}\Big\|_{\mathcal H}^2
	=
	\sum_{k=-(n-1)}^{n-1}\Big(1-\frac{|k|}{n}\Big)_+\,
	\tr\big((I-P_m)C_k^{(\star)}(I-P_m)\big).
	\]
	As $n\to\infty$ the right-hand side converges to
	$\tr\big((I-P_m)\Sigma^{(\star)}(I-P_m)\big)$, which goes to $0$ as $m\to\infty$
	since $\Sigma^{(\star)}$ is trace--class (\textnormal{LRV}). Thus the tail is uniformly small,
	yielding tightness. Combining Step~1 and Step~2 gives the ${\mathcal H}$--CLT.
	
	The statement for the total process follows from Lemma~\ref{lem:parity}.
\end{proof}

\subsubsection{Operator-form long-run covariance and FCLT}\label{subsec:opFCLT}

Fix $\star\in\{\mathrm{size},\mathrm{loc}\}$.
Assume weak stationarity for $(Z_i^{(\star)})$ and $\E\|Z_0^{(\star)}\|_{\H}^2<\infty$,
where $\H:=L^2(\mathbb S^{d-1},\sigma)$.
For $k\ge1$, write $\mathcal F_i:=\sigma(Z_i^{(\star)})$ and define the classical HGR maximal correlation
\[
\overline\rho(k)
\ :=\ \sup\left\{ \Corr(F,G):\ F\in L^2_0(\mathcal F_0),\ G\in L^2_0(\mathcal F_k)\right\}\in[0,1].
\]
We assume throughout this subsection
\begin{equation}\label{ass:rho-bar}
	\sum_{k=1}^{\infty}\overline\rho(k)\ <\ \infty.
	\tag{Mix-op}
\end{equation}

\begin{proposition}[Operator representation of the long-run covariance]\label{prop:LRV-repr}
	Let $\H:=L^2(\mathbb S^{d-1},\sigma)$ and, for $\star\in\{\mathrm{size},\mathrm{loc}\}$, set
	$Z_i^{(\star)}:=\widetilde\Phi^{(\star)}_{X_i}\in\H$ with $\E Z_i^{(\star)}=0$.
	For $f,g\in\H$, consider the bilinear form
	\[
	\mathfrak K_\star(f,g)
	\ :=\ \sum_{k\in\mathbb Z}\Cov\!\big(\langle Z_0^{(\star)},f\rangle_{\H},\ \langle Z_k^{(\star)},g\rangle_{\H}\big),
	\]
	whenever the series is absolutely convergent. Then:
	\begin{enumerate}
		\item There exists a unique bounded, self-adjoint, positive semidefinite operator
		$\mathcal K_\star:\H\to\H$ such that
		\[
		\langle \mathcal K_\star f,\ g\rangle_{\H}\ =\ \mathfrak K_\star(f,g),
		\qquad f,g\in\H.
		\]
		\item If in \eqref{eq:lrv-kernel} the long-run covariance is defined via kernels by
		$K_k^{(\star)}(u,v):=\Cov\big(Z_0^{(\star)}(u),Z_k^{(\star)}(v)\big)$,
		$(C_k^{(\star)}f)(u)=\int K_k^{(\star)}(u,v)f(v)\,d\sigma(v)$ and
		$\Sigma^{(\star)}:=\sum_{k\in\mathbb Z}C_k^{(\star)}$, then
		\[
		\langle \Sigma^{(\star)} f,\ g\rangle_{\H}\ =\ \mathfrak K_\star(f,g)
		\quad\text{for all }f,g\in\H,
		\]
		hence $\mathcal K_\star=\Sigma^{(\star)}$.
	\end{enumerate}
\end{proposition}

\begin{proof}
	For (1): $\mathfrak K_\star$ is a bounded symmetric positive bilinear form (absolute convergence and boundedness
	follow from the mixing/second-moment assumptions used in the paper). By the Riesz representation theorem on the Hilbert
	space $\H$, there exists a unique bounded self-adjoint PSD operator $\mathcal K_\star$ such that
	$\langle \mathcal K_\star f,g\rangle_{\H}=\mathfrak K_\star(f,g)$ for all $f,g\in\H$.
	
	For (2): Using the kernel representation of $C_k^{(\star)}$ and Fubini/Parseval,
	\[
	\sum_{k\in\mathbb Z}\langle C_k^{(\star)}f,\ g\rangle_{\H}
	=\sum_{k\in\mathbb Z}\E\big[\langle Z_0^{(\star)},f\rangle_{\H}\langle Z_k^{(\star)},g\rangle_{\H}\big]
	=\mathfrak K_\star(f,g).
	\]
	Thus $\langle \Sigma^{(\star)}f,g\rangle_{\H}=\mathfrak K_\star(f,g)$ for all $f,g$, which forces
	$\mathcal K_\star=\Sigma^{(\star)}$ by uniqueness in (1).
\end{proof}

\begin{proposition}[Boundedness, PSD, and trace-class]\label{prop:Kstar-trace}
	Assume \eqref{ass:rho-bar} and $\E\|Z_0^{(\star)}\|_{\H}^2<\infty$. Then for $\star\in\{\mathrm{size},\mathrm{loc}\}$:
	\begin{enumerate}
		\item[(i)] $\mathfrak K_\star$ is well-defined and bounded; the induced operator $\mathcal K_\star$ exists,
		is self-adjoint and positive semidefinite. Moreover, for all $f,g\in \H$,
		\[
		|\mathfrak K_\star(f,g)|
		\ \le\ \E\|Z_0^{(\star)}\|_{\H}^2\Big(1+2\sum_{k\ge1}\overline\rho(k)\Big)\,\|f\|_{\H}\,\|g\|_{\H}.
		\]
		\item[(ii)] $\mathcal K_\star$ is trace-class and
		\[
		\tr\,\mathcal K_\star
		\ =\ \sum_{k\in\mathbb Z}\E\langle Z_0^{(\star)},Z_k^{(\star)}\rangle_{\H}
		\ \le\ \E\|Z_0^{(\star)}\|_{\H}^2\Big(1+2\sum_{k\ge1}\overline\rho(k)\Big)\ <\ \infty.
		\]
	\end{enumerate}
\end{proposition}
\begin{proof}
	Fix $f,g\in \H$ with $\|f\|_{\H}=\|g\|_{\H}=1$ (the general case follows by homogeneity).
	Put $U_k:=\langle Z_0^{(\star)},f\rangle_{\H}$ and $V_k:=\langle Z_k^{(\star)},g\rangle_{\H}$.
	Then $U_k\in L^2_0(\mathcal F_0)$, $V_k\in L^2_0(\mathcal F_k)$, and by the HGR bound,
	for $k\ge1$,
	\[
	|\Cov(U_k,V_k)|
	\ \le\ \overline\rho(k)\,\sqrt{\Var(U_k)\,\Var(V_k)}
	\ \le\ \overline\rho(k)\,\E\|Z_0^{(\star)}\|_{\H}^2,
	\]
	because $\Var(\langle Z_i^{(\star)},h\rangle_{\H})\le \E\|Z_i^{(\star)}\|_{\H}^2\|h\|_{\H}^2
	=\E\|Z_0^{(\star)}\|_{\H}^2$.
	For $k=0$,
	\[
	|\Cov(U_0,V_0)|
	\le \sqrt{\Var(U_0)\Var(V_0)}
	\le \E\|Z_0^{(\star)}\|_{\H}^2.
	\]
	Therefore $\sum_{k\in\mathbb Z}\Cov(U_k,V_k)$ is absolutely convergent and
	\[
	|\mathfrak K_\star(f,g)|
	\ \le\ \E\|Z_0^{(\star)}\|_{\H}^2\Big(1+2\sum_{k\ge1}\overline\rho(k)\Big).
	\]
	By polarization, $\mathfrak K_\star$ is a bounded symmetric bilinear form.
	Positivity follows from
	\[
	\mathfrak K_\star(f,f)
	\ =\ \sum_{k\in\mathbb Z}\Cov\!\big(\langle Z_0^{(\star)},f\rangle_{\H},\ \langle Z_k^{(\star)},f\rangle_{\H}\big)
	\ \ge\ 0,
	\]
	interpretable as the long-run variance of the scalar weakly stationary sequence
	$\{\langle Z_i^{(\star)},f\rangle_{\H}\}$.
	Hence, by the Riesz representation theorem, there exists a unique bounded self-adjoint PSD operator
	$\mathcal K_\star$ with $\langle \mathcal K_\star f,g\rangle_{\H}=\mathfrak K_\star(f,g)$.
	
	For the trace claim, let $(e_j)_{j\ge1}$ be any orthonormal basis of $\H$. Using Parseval and Fubini,
	\begin{align*}
		\tr\,\mathcal K_\star
		\ =\ \sum_{j\ge1}\langle \mathcal K_\star e_j,e_j\rangle_{\H}
		& =\ \sum_{k\in\mathbb Z}\sum_{j\ge1}\Cov\!\big(\langle Z_0^{(\star)},e_j\rangle_{\H},\ \langle Z_k^{(\star)},e_j\rangle_{\H}\big)\\
		& =\ \sum_{k\in\mathbb Z}\E\langle Z_0^{(\star)},Z_k^{(\star)}\rangle_{\H}.
	\end{align*}
	For $k\ge1$,
	\begin{align*}
		\sum_{j\ge1}\Big|\Cov\!\big(\langle Z_0^{(\star)},e_j\rangle_{\H},\ \langle Z_k^{(\star)},e_j\rangle_{\H}\big)\Big|
		& \le\ \overline\rho(k)\sum_{j\ge1}\sqrt{\Var\langle Z_0^{(\star)},e_j\rangle_{\H}\,\Var\langle Z_k^{(\star)},e_j\rangle_{\H}}\\
		& =\ \overline\rho(k)\,\E\|Z_0^{(\star)}\|_{\H}^2,
	\end{align*}
	since $\Var\langle Z_0^{(\star)},e_j\rangle_{\H}=\Var\langle Z_k^{(\star)},e_j\rangle_{\H}$ by stationarity and
	$\sum_j\Var\langle Z_0^{(\star)},e_j\rangle_{\H}=\E\|Z_0^{(\star)}\|_{\H}^2$.
	The bound for $k=0$ is the same with $\overline\rho(0)=1$.
	Summing over $k$ gives the stated upper bound and absolute convergence, hence $\mathcal K_\star$ is trace-class.
\end{proof}

\begin{definition}[Hilbert-valued Brownian motion with covariance $\mathcal K_\star$]\label{def:Hbm}
	An $\H$-valued process $B_\star=\{B_\star(t):t\ge0\}$ is an $\H$-Brownian motion with covariance
	operator $\mathcal K_\star$ if for all $f,g\in \H$ and $s,t\ge0$,
	\[
	\E\big[\langle B_\star(s),f\rangle_{\H}\langle B_\star(t),g\rangle_{\H}\big]
	\ =\ \min\{s,t\}\,\langle \mathcal K_\star f,\ g\rangle_{\H}.
	\]
	Equivalently, if $\mathcal K_\star e_r=\lambda_r e_r$ with $\sum_r\lambda_r<\infty$, then
	\[
	B_\star(t)\ \stackrel{d}{=}\ \sum_{r\ge1}\sqrt{\lambda_r}\,\beta_r(t)\,e_r,
	\]
	with i.i.d.\ standard Brownian motions $\{\beta_r\}$.
\end{definition}

\begin{theorem}[Operator-form FCLT / invariance principle]\label{thm:OFCLT}
	Assume \eqref{ass:rho-bar} and $\E\|Z_0^{(\star)}\|_{\H}^{2+\delta}<\infty$ for some $\delta>0$.
	For $\star\in\{\mathrm{size},\mathrm{loc}\}$, define the step process in $D([0,1];\H)$
	\[
	S_n^{(\star)}(t)\ :=\ \frac{1}{\sqrt n}\sum_{i=1}^{\lfloor nt\rfloor} Z_i^{(\star)},\qquad t\in[0,1].
	\]
	Then $S_n^{(\star)}\Rightarrow B_\star$ in $D([0,1];\H)$, where $B_\star$ is an $\H$-Brownian motion with
	covariance operator $\mathcal K_\star$ from Proposition~\ref{prop:Kstar-trace}.
	Consequently, for $Z_i^{(\mathrm{tot})}:=Z_i^{(\mathrm{size})}+Z_i^{(\mathrm{loc})}$,
	\[
	S_n^{(\mathrm{tot})}\ \Rightarrow\ B_{\mathrm{size}}+B_{\mathrm{loc}},
	\qquad
	\mathcal K_{\mathrm{tot}}=\mathcal K_{\mathrm{size}}+\mathcal K_{\mathrm{loc}},
	\]
	and $B_{\mathrm{size}}$ and $B_{\mathrm{loc}}$ are orthogonal Gaussian processes due to even/odd orthogonality.
\end{theorem}
\begin{proof}
	\emph{Step 1: finite-dimensional distributions.}
	Fix $m\ge1$, $0\le t_1<\cdots<t_m\le1$, and $h_1,\dots,h_m\in \H$.
	Let $\Delta_{n,j}:=S_n^{(\star)}(t_j)-S_n^{(\star)}(t_{j-1})$ with $t_0:=0$.
	For $c=(c_1,\dots,c_m)\in\mathbb R^m$, consider the scalar array
	\[
	Y_{n,i}\ :=\ \sum_{j=1}^m c_j\,\Big\langle Z_i^{(\star)}, h_j\Big\rangle_{\H}\,
	\Big(\mathbf 1\{i\le \lfloor nt_j\rfloor\}-\mathbf 1\{i\le \lfloor nt_{j-1}\rfloor\}\Big),
	\qquad 1\le i\le n.
	\]
	Then $\sum_{i=1}^n Y_{n,i}=\sum_{j=1}^m c_j\,\langle \Delta_{n,j},h_j\rangle_{\H}$ and each $Y_{n,i}\in L^2_0(\mathcal F_i)$.
	By weak stationarity and \eqref{ass:rho-bar}, the array $\{Y_{n,i}\}$ is uniformly $\rho$-mixing with coefficients bounded by $\overline\rho(k)$.
	Moreover, Proposition~\ref{prop:Kstar-trace} identifies the limiting covariance:
	\[
	\sum_{i=1}^n \E Y_{n,i}^2\ \to\ \sum_{j=1}^m\sum_{j'=1}^m c_j c_{j'}\,\min\{t_j,t_{j'}\}
	\,\langle \mathcal K_\star h_j,\ h_{j'}\rangle_{\H}.
	\]
	Under $\E\|Z_0^{(\star)}\|_{\H}^{2+\delta}<\infty$, a standard big-block/small-block argument yields the Lindeberg condition for $\{Y_{n,i}\}$:
	\[
	\sum_{i=1}^n \E\{Y_{n,i}^2\,\mathbf 1(|Y_{n,i}|>\varepsilon)\}\ \to\ 0\qquad(\forall\,\varepsilon>0).
	\]
	Therefore, a $\rho$-mixing CLT for triangular arrays gives
	$\sum_{i=1}^n Y_{n,i}\Rightarrow \mathcal N(0,\Sigma_c)$ with $\Sigma_c$ equal to the displayed limit.
	By Cram\'er--Wold, the vector of increments converges to a centered Gaussian vector with the corresponding covariance,
	hence the finite-dimensional distributions of $S_n^{(\star)}$ converge to those of an $\H$-Brownian motion $B_\star$
	with covariance operator $\mathcal K_\star$.
	
	\emph{Step 2: tightness in $D([0,1];\H)$.}
	Let $K_\rho:=\E\|Z_0^{(\star)}\|_{\H}^2\big(1+2\sum_{k\ge1}\overline\rho(k)\big)$.
	Stationarity and \eqref{ass:rho-bar} yield the block second-moment bound: for all integers $m\ge0,\ell\ge1$,
	\[
	\E\Big\|\sum_{i=m+1}^{m+\ell} Z_i^{(\star)}\Big\|_{\H}^2\ \le\ K_\rho\,\ell.
	\]
	Thus for $0\le s\le t\le1$,
	\[
	\E\big\|S_n^{(\star)}(t)-S_n^{(\star)}(s)\big\|_{\H}^2
	\ =\ \frac{1}{n}\,\E\Big\|\sum_{i=\lfloor ns\rfloor+1}^{\lfloor nt\rfloor} Z_i^{(\star)}\Big\|_{\H}^2
	\ \le\ K_\rho\,|t-s|+O\!\left(\frac{1}{n}\right),
	\]
	which yields tightness in $D([0,1];\H)$ by standard criteria for separable Hilbert-space valued processes.
	
	\emph{Conclusion.}
	Finite-dimensional convergence and tightness imply $S_n^{(\star)}\Rightarrow B_\star$ in $D([0,1];\H)$.
	
	For the total process, note that $Z_i^{(\mathrm{tot})}=Z_i^{(\mathrm{size})}+Z_i^{(\mathrm{loc})}$ and the even/odd orthogonality yields
	$\mathcal K_{\mathrm{tot}}=\mathcal K_{\mathrm{size}}+\mathcal K_{\mathrm{loc}}$ and orthogonality of the Gaussian limits.
\end{proof}

\begin{lemma}[Linear interpolation is negligible]\label{lem:interp}
	Assume $\E\|Z_0^{(\star)}\|_{\H}^{2+\delta}<\infty$ for some $\delta>0$.
	Define the step and polygonal versions, for $t\in[0,1]$,
	\[
	S_n^{(\star)}(t)=\frac{1}{\sqrt n}\sum_{i=1}^{\lfloor nt\rfloor}Z_i^{(\star)},
	\qquad
	\mathbb S_n^{(\star)}(t)=S_n^{(\star)}\!\Big(\frac{\lfloor nt\rfloor}{n}\Big)
	+\big(nt-\lfloor nt\rfloor\big)\frac{Z^{(\star)}_{\lfloor nt\rfloor+1}}{\sqrt n}.
	\]
	Then $\sup_{t\in[0,1]}\|\mathbb S_n^{(\star)}(t)-S_n^{(\star)}(t)\|_{\H}\xrightarrow{\ \mathbb P\ }0$.
\end{lemma}

\begin{proof}
	For every $t$,
	\[
	\|\mathbb S_n^{(\star)}(t)-S_n^{(\star)}(t)\|_{\H}
	\le \frac{1}{\sqrt n}\,\max_{1\le i\le n+1}\|Z_i^{(\star)}\|_{\H}.
	\]
	Fix $\varepsilon>0$. By the union bound and Markov with exponent $2+\delta$,
	\[
	\Pr\!\left(\max_{1\le i\le n+1}\|Z_i^{(\star)}\|_{\H}>\varepsilon\sqrt n\right)
	\ \le\ (n+1)\,\frac{\E\|Z_0^{(\star)}\|_{\H}^{2+\delta}}{\varepsilon^{2+\delta} n^{1+\delta/2}}
	\ \xrightarrow[n\to\infty]{}\ 0,
	\]
	which yields the claim.
\end{proof}

\begin{remark}[Compatibility with $\rho_{\max}$ and a diagonal-dominance route]\label{rem:compat-rhomax-op}
	By construction, $\overline\rho(k)$ is the classical HGR coefficient between
	$\sigma(Z_0^{(\star)})$ and $\sigma(Z_k^{(\star)})$, hence for any $u\in\Sph$,
	\[
	\rho\!\big(Z_0^{(\star)}(u),Z_k^{(\star)}(u)\big)\ \le\ \overline\rho(k).
	\]
	Therefore $\rho_{\max}(k)\le \overline\rho(k)$ and $\sum_k\overline\rho(k)<\infty$ implies
	$\sum_k\rho_{\max}(k)<\infty$ (the assumption used earlier for LLN/SLLN).
	
	If one wishes to formulate Theorem~\ref{thm:OFCLT} directly under $\sum_k\rho_{\max}(k)<\infty$,
	it suffices to assume a mild diagonal dominance of
	$C_k(u,v):=\Cov(Z_0^{(\star)}(u),Z_k^{(\star)}(v))$, e.g.
	\[
	\int_{\Sph}\!\int_{\Sph}\!\big|C_k(u,v)\big|\,d\sigma(u)\,d\sigma(v)
	\ \lesssim\ \int_{\Sph}\!\big|C_k(u,u)\big|\,d\sigma(u),
	\qquad k\ge1,
	\]
	which holds, for instance, when cross-direction correlations decay integrably in $|u-v|$.
	Under this condition, the bounds above go through with $\rho_{\max}(k)$ in place of $\overline\rho(k)$.
\end{remark}

\begin{corollary}[Componentwise and total FCLT]\label{cor:FCLT-comp-total}
	Under \textnormal{(M2+)}, \textnormal{(Mix)}, and \textnormal{(LRV)}, in $D([0,1];\H)$ we have
	\[
	S_n^{(\mathrm{size})}\Rightarrow \mathbb B^{(\mathrm{size})},\qquad
	S_n^{(\mathrm{loc})}\Rightarrow \mathbb B^{(\mathrm{loc})},
	\]
	where the limits are $\H$-valued Brownian motions with
	$\E\langle \mathbb B^{(\star)}(s),f\rangle\,\langle \mathbb B^{(\star)}(t),g\rangle
	=\min\{s,t\}\,\langle \mathcal K_\star f,g\rangle$.
	Consequently, by even/odd orthogonality,
	\[
	S_n^{(\mathrm{tot})}\Rightarrow \mathbb B^{(\mathrm{size})}+\mathbb B^{(\mathrm{loc})},
	\qquad \mathcal K_{\mathrm{tot}}=\mathcal K_{\mathrm{size}}+\mathcal K_{\mathrm{loc}}.
	\]
	Moreover, the polygonal versions $\mathbb S_n^{(\star)}$ enjoy the same limits in $C([0,1];\H)$.
\end{corollary}

\begin{proof}
	Theorem~\ref{thm:OFCLT} gives the step-process limits. Lemma~\ref{lem:interp} upgrades them to polygonal limits.
\end{proof}

\subsubsection{Law of the Iterated Logarithm in $\H$}

Put $a_n:=\sqrt{2n\log\log n}$.

\begin{theorem}[Hilbert--space LIL]\label{thm:LIL}
	Assume \textnormal{(M$2+$)} and, for some $\delta>0$,
	\[
	\sum_{k\ge1}k^{1/2}\,\rho_{\max}(k)^{\delta/(2+\delta)}<\infty,
	\]
	and \textnormal{(LRV)}. Then, almost surely,
	\[
	\limsup_{n\to\infty}\frac{\|T_n^{(\star)}\|_{\mathcal H}}{a_n}
	=\|\Sigma^{(\star)}\|_{\mathrm{op}}^{1/2},\qquad \star\in\{\mathrm{size},\mathrm{loc}\},
	\]
	and the cluster set of $\{T_n^{(\star)}/a_n\}$ equals the closed unit ball of
	$\Sigma^{(\star)\,1/2}({\mathcal H})$.
	Moreover,
	\[
	\limsup_{n\to\infty}\frac{\|T_n^{(\mathrm{tot})}\|_{\mathcal H}}{a_n}
	=\max\!\Big\{\|\Sigma^{(\mathrm{size})}\|_{\mathrm{op}}^{1/2},\ \|\Sigma^{(\mathrm{loc})}\|_{\mathrm{op}}^{1/2}\Big\}.
	\]
	Here $\|\cdot \|_{\mathrm{op}}$ denotes the operator norm (largest eigenvalue) of the covariance operator.
\end{theorem}
\begin{proof}
	\emph{(Sketch via ASIP and transfer of the Brownian LIL).}
	We outline the standard route: establish an almost sure invariance principle (ASIP)
	for the $\H$--valued partial sums, then transfer the classical Hilbert--space Brownian LIL/Strassen compact LIL.
	
	\smallskip
	\noindent
	\emph{Step 1: Blocking decomposition.}
	Choose increasing integers $(b_j)_{j\ge1}$ (big--block lengths) and $(s_j)_{j\ge1}$ (gaps) such that
	$b_j\to\infty$, $s_j\to\infty$, $s_j/b_j\to0$, and
	\[
	\sum_{j\ge1}\frac{s_j}{\sqrt{\sum_{i\le j} b_i}}\ <\ \infty.
	\]
	Let $B_j$ be the $j$th big block and $G_j$ the $j$th gap block.
	Write
	\[
	T_n^{(\star)} \ =\ \sum_{j:\,B_j\subset\{1,\dots,n\}} \sum_{i\in B_j} Z_i^{(\star)}
	\ +\ \sum_{j:\,G_j\subset\{1,\dots,n\}} \sum_{i\in G_j} Z_i^{(\star)}
	\ +\ R_n,
	\]
	where $R_n$ is the boundary remainder.
	
	\smallskip
	\noindent
	\emph{Step 2: Gaps are negligible at the LIL scale.}
	Using \textnormal{(M$2+$)} and the summability
	$\sum_{k\ge1}k^{1/2}\rho_{\max}(k)^{\delta/(2+\delta)}<\infty$,
	one applies a Davydov/Rio--type covariance bound for $\rho$--mixing to obtain moment estimates for
	gap partial sums, and then a Borel--Cantelli argument yields
	\[
	\sum_{j\ge1}\Big\|\sum_{i\in G_j} Z_i^{(\star)}\Big\|_{\H}
	\ =\ o\!\big(a_n\big)\qquad\text{a.s.},
	\]
	and similarly $R_n=o(a_n)$ almost surely.
	
	\smallskip
	\noindent
	\emph{Step 3: ASIP for big blocks (Gaussian coupling).}
	For the big--block sums, one constructs on an extension of the probability space an $\H$--valued Brownian motion
	$\mathbb B^{(\star)}$ with covariance operator $\Sigma^{(\star)}$ (guaranteed by \textnormal{(LRV)}),
	and couples it with the sequence of big--block partial sums (Philipp--Stout/Skorokhod-type coupling) so that
	\[
	\sup_{m\le n}
	\Bigg\|
	\sum_{j\le m}\sum_{i\in B_j} Z_i^{(\star)}
	-\mathbb B^{(\star)}\!\Big(\Var\!\Big(\sum_{j\le m}\sum_{i\in B_j} Z_i^{(\star)}\Big)\Big)
	\Bigg\|_{\H}
	\ =\ o\!\big(a_n\big)
	\qquad\text{a.s.}
	\]
	Combining with Step~2 yields the ASIP
	\[
	T_n^{(\star)} \ =\ \mathbb B^{(\star)}(n)\ +\ o\!\big(a_n\big)\qquad\text{a.s.}
	\]
	
	\smallskip
	\noindent
	\emph{Step 4: Transfer of the LIL and cluster set.}
	The Hilbert--space LIL for Brownian motion gives
	\[
	\limsup_{t\to\infty}\frac{\|\mathbb B^{(\star)}(t)\|_{\H}}{\sqrt{2t\log\log t}}
	\ =\ \|\Sigma^{(\star)}\|_{\op}^{1/2},
	\]
	and Strassen's compact LIL identifies the cluster set as the closed unit ball of
	$\Sigma^{(\star)1/2}(\H)$.
	Since the ASIP error is $o(a_n)$, the same limsup and the same cluster set hold for $T_n^{(\star)}$.
	
	\smallskip
	\noindent
	\emph{Total component.}
	Using the orthogonal decomposition $\H=\H_{\even}\oplus \H_{\odd}$ with
	$T_n^{(\mathrm{size})}\in\H_{\even}$ and $T_n^{(\mathrm{loc})}\in\H_{\odd}$,
	we have
	\[
	\|T_n^{(\mathrm{tot})}\|_{\H}^2
	=\|T_n^{(\mathrm{size})}\|_{\H}^2+\|T_n^{(\mathrm{loc})}\|_{\H}^2.
	\]
	Hence the LIL constant for the total norm is the maximum of the two radii, i.e.
	\[
	\limsup_{n\to\infty}\frac{\|T_n^{(\mathrm{tot})}\|_{\H}}{a_n}
	=\max\!\Big\{\|\Sigma^{(\mathrm{size})}\|_{\op}^{1/2},\ \|\Sigma^{(\mathrm{loc})}\|_{\op}^{1/2}\Big\}.
	\]
\end{proof}

\subsubsection*{Remarks and corollaries}

\begin{itemize}
	\item[\textbullet] \textbf{Scalar projections.}
	For any fixed $h\in \mathcal H$, the real--valued sums satisfy
	$\frac{1}{\sqrt n}\sum_{i=1}^n\langle Z_i^{(\star)},h\rangle\Rightarrow \mathcal N\big(0,\sigma_\star(h)^2\big)$ with
	$\sigma_\star(h)^2=\langle \Sigma^{(\star)}h,h\rangle$; LIL and FCLT follow similarly.
	
	\item[\textbullet] \textbf{Checkable (LRV).}
	By Lemma~\ref{lem:lrv}, \textnormal{(M$2+$)}+\textnormal{(Mix)} imply $\Sigma^{(\star)}$ is trace--class.
	If desired, the stronger $\sum_k\rho_{\max}(k)<\infty$ can be assumed for simpler verification.
	
	\item[\textbullet] \textbf{Parity decoupling.}
	All limit objects are block--diagonal on ${\mathcal H}_{\mathrm{even}}\oplus \mathcal{H}_{\mathrm{odd}}$ by Lemma~\ref{lem:parity};
	hence ``total'' limits decompose as sums of orthogonal Gaussian components supported on the two orthogonal subspaces.
	
	\item[\textbullet] \textbf{Ergodic/LLN context.}
	The integrated Chebyshev bounds and WLLN/SLLN in $\mathcal H$ established earlier are repeatedly used in Lemma~\ref{lem:increments}
	and in the tightness arguments (we only recalled the needed ingredients).
\end{itemize}

\section{Estimation and simulation study}
\label{sec:simulation}

This section illustrates, by simulation, how the proposed even--odd
support--function decomposition separates \emph{size} and \emph{location}
dependence, and how it improves upon Steiner-point--based correlation
indices.
The focus is on conceptual validation through carefully designed
dependence scenarios rather than algorithmic optimization.

\subsection{Simulation model}

All simulations are conducted in $\mathbb{R}^2$.
Each random set $X_i$ is a \emph{non--centrally symmetric} convex polygon,
specifically an asymmetric triangle.
This choice is intentional: for such shapes, the odd component of the
support function contains directional information beyond what is captured
by a single point-valued summary such as the Steiner point.

Let $V_i=\{v_{i1},v_{i2},v_{i3}\}\subset\mathbb{R}^2$ denote the vertices of
$X_i$.
The support function is
\[
h_{X_i}(u)=\max_{k=1,2,3}\langle u,v_{ik}\rangle,
\qquad u\in\mathbb{S}^1.
\]
Its even and odd components are
\[
W_{X_i}(u)=\tfrac12\big(h_{X_i}(u)+h_{X_i}(-u)\big),\qquad
C_{X_i}(u)=\tfrac12\big(h_{X_i}(u)-h_{X_i}(-u)\big).
\]

The triangles evolve over time through weakly stationary AR(1) dynamics in:
(i) translation (center),
(ii) shape parameters (side lengths and skewness),
and (iii) orientation.
All innovation sequences are centered, independent across components, and
have finite second moments.
This ensures weak stationarity and mixing conditions compatible with the
theory developed earlier.

In addition to the raw odd component $C_{X_i}$, we also consider its
\emph{Steiner--centered (residual) version}
\[
C_{X_i}^{\mathrm{res}}(u)
=
C_{X_i}(u)-\langle s(X_i),u\rangle,
\]
where $s(X_i)$ denotes the Steiner point.
As discussed in Section~\ref{subsec:steiner}, this residual removes the
linear part of the odd component induced by the Steiner point and isolates
directional location fluctuations.

\subsection{Dependence scenarios}

We consider four scenarios, denoted S1--S4, designed to jointly highlight
(i) the ability of the even--odd decomposition to disentangle size and
location dependence, and
(ii) the limitations of Steiner-point--based correlation indices.

\paragraph{S1: Sign flip}
Let
\[
Y_i=-X_i .
\]
Then the size components coincide while the location components reverse
sign. Consequently,
\[
\Corr_{\mathrm{size}}(X,Y)\approx 1,\qquad
\Corr_{\mathrm{loc}}^{\mathrm{res}}(X,Y)\approx -1,
\]
and the Steiner-point correlation is close to $-1$.
This scenario serves as a baseline sanity check.

\paragraph{S2: Partial location dependence}
Let $Y_i$ share the same translation process as $X_i$, while its shape and orientation evolve independently.
In this case, the Steiner points coincide, so the Steiner-point correlation is close to one.
However, the residual location correlation
$\Corr_{\mathrm{loc}}^{\mathrm{res}}(X,Y)$ is substantially smaller, reflecting the fact that only the linear (point-valued) part of the odd component is shared, whereas higher-order directional fluctuations remain largely independent.

\paragraph{S3: Size dependence with directional asymmetry}
Let $Y_i$ share the same shape and orientation as $X_i$, while its translation process is independent.
As expected, the Steiner-point correlation is close to zero. Nevertheless, the residual location correlation $\Corr_{\mathrm{loc}}^{\mathrm{res}}(X,Y)$ is clearly positive.
This is due to the use of asymmetric shapes: although the centers are independent, shared shape dynamics induce coherent directional fluctuations in the odd component that are invisible to the Steiner point.

\paragraph{S4: Steiner--invisible location dependence (sensitivity analysis)}
This scenario is designed as a \emph{continuous analogue of S3}.
Let $Z_i$ be an independent copy of $X_i$ (with identical marginal
geometry but independent dynamics), and define
\[
Y_i=-\alpha X_i+(1-\alpha)Z_i,\qquad \alpha\in[0,1].
\]
By linearity of the Steiner point,
\[
s(Y_i)=-\alpha s(X_i)+(1-\alpha)s(Z_i),
\]
so that $s(Y_i)$ is independent of $s(X_i)$ for all $\alpha$.
Hence the Steiner-point correlation remains close to zero for all
$\alpha$.

In contrast, the residual odd component satisfies
\[
C_{Y_i}^{\mathrm{res}}
=
-\alpha C_{X_i}^{\mathrm{res}}+(1-\alpha)C_{Z_i}^{\mathrm{res}},
\]
so that $\Corr_{\mathrm{loc}}^{\mathrm{res}}(X,Y)$ varies continuously with
$\alpha$.
This scenario isolates a form of directional location dependence that is
completely invisible to Steiner-point--based summaries.

\subsection{Estimation and implementation}

For each scenario we generate a time series of length $n$ and sample
$M$ directions uniformly on $\mathbb{S}^1$.
Direction-wise covariances are computed with Bessel correction and
integrated by Monte--Carlo averaging.
We report
\[
\widehat{\Corr}_{\mathrm{size}},\qquad
\widehat{\Corr}_{\mathrm{loc}}^{\mathrm{res}},
\]
together with the Steiner-point correlation for comparison.

The asymptotic properties of these estimators under weak dependence follow
from the Hilbert--space limit theorems developed earlier and are summarized
in \ref{app:Consistency_simulation}.
Here we focus on qualitative behavior.

\subsection{Results}

Figure~\ref{fig:sim_S1S3_triangles} reports the estimated correlations for
scenarios S1--S3 based on asymmetric triangles.
Scenario~S1 serves as a sanity check: the sign flip $Y=-X$ yields
$\Corr_{\mathrm{size}}\approx 1$ and a negative residual location
correlation, which is also captured by the Steiner-point correlation.

\begin{figure}[ht]
	\centering
	\includegraphics[width=0.9\linewidth]{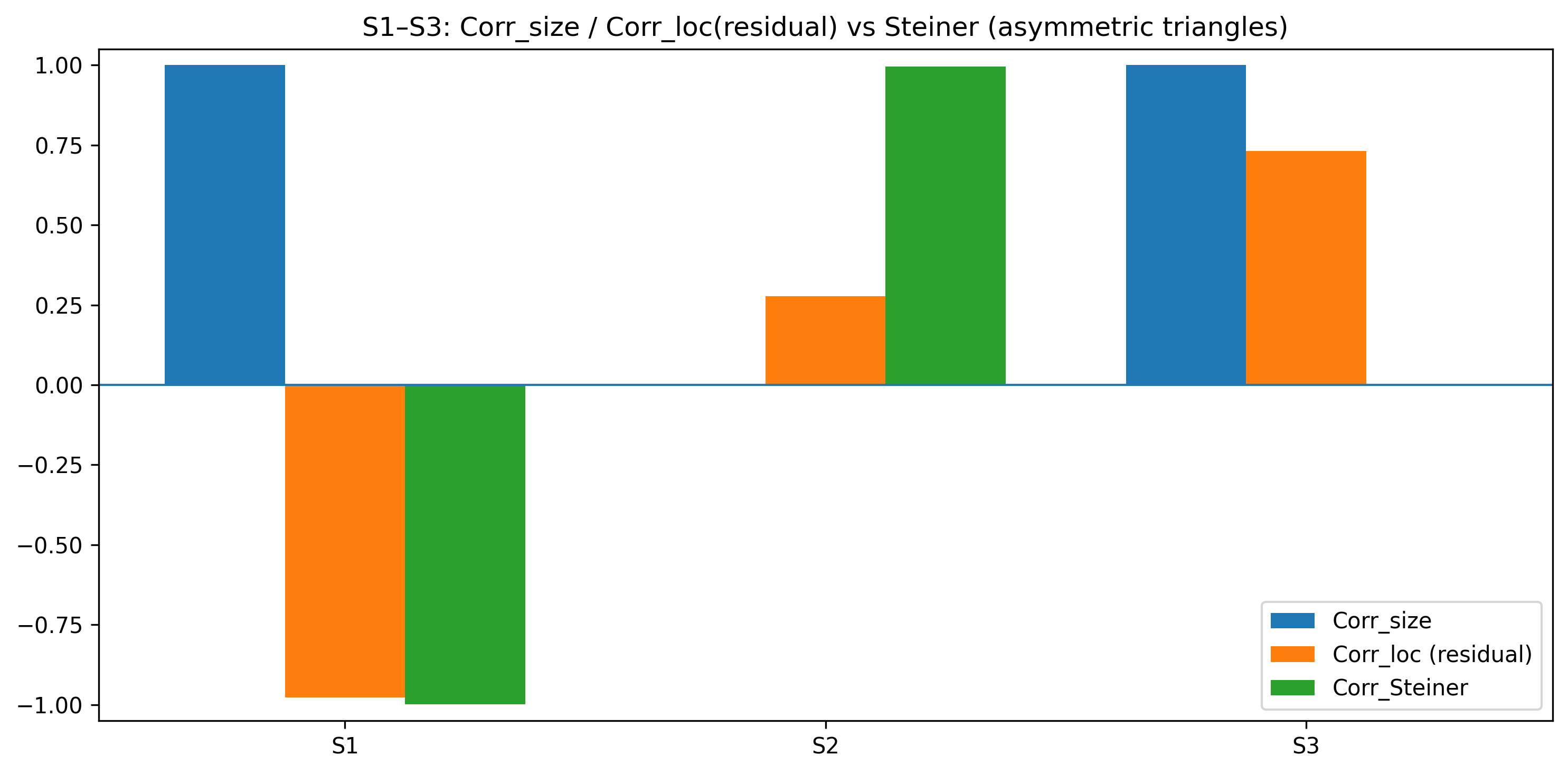}
	\caption{Scenarios S1--S3 with asymmetric triangles.
		Residual location correlation $\Corr_{\mathrm{loc}}^{\mathrm{res}}$
		separates clearly from the Steiner-point correlation, particularly in
		Scenario~S3 where Steiner fails to detect strong size dependence.}
	\label{fig:sim_S1S3_triangles}
\end{figure}

Scenarios~S2 and S3 reveal the essential difference between
Steiner-based summaries and the proposed even--odd decomposition.
In Scenario~S2, the sets share the same center but have independent shape
dynamics.
Accordingly, the Steiner-point correlation is close to one, while the
residual location correlation is substantially smaller, indicating that
most of the apparent dependence is explained by the common translation
and not by directional location effects.
In contrast, Scenario~S3 exhibits strong size dependence together with
pronounced residual location dependence, even though the Steiner-point
correlation is close to zero.
This shows that directional location dependence may persist and remain
detectable in the odd support-function profile, despite being completely
invisible to any point-valued summary such as the Steiner point.

Figure~\ref{fig:sim_S4_triangles} displays Scenario~S4, which interpolates
between independence and perfect size dependence while keeping the
Steiner point uninformative.

\begin{figure}[ht]
	\centering
	\includegraphics[width=0.9\linewidth]{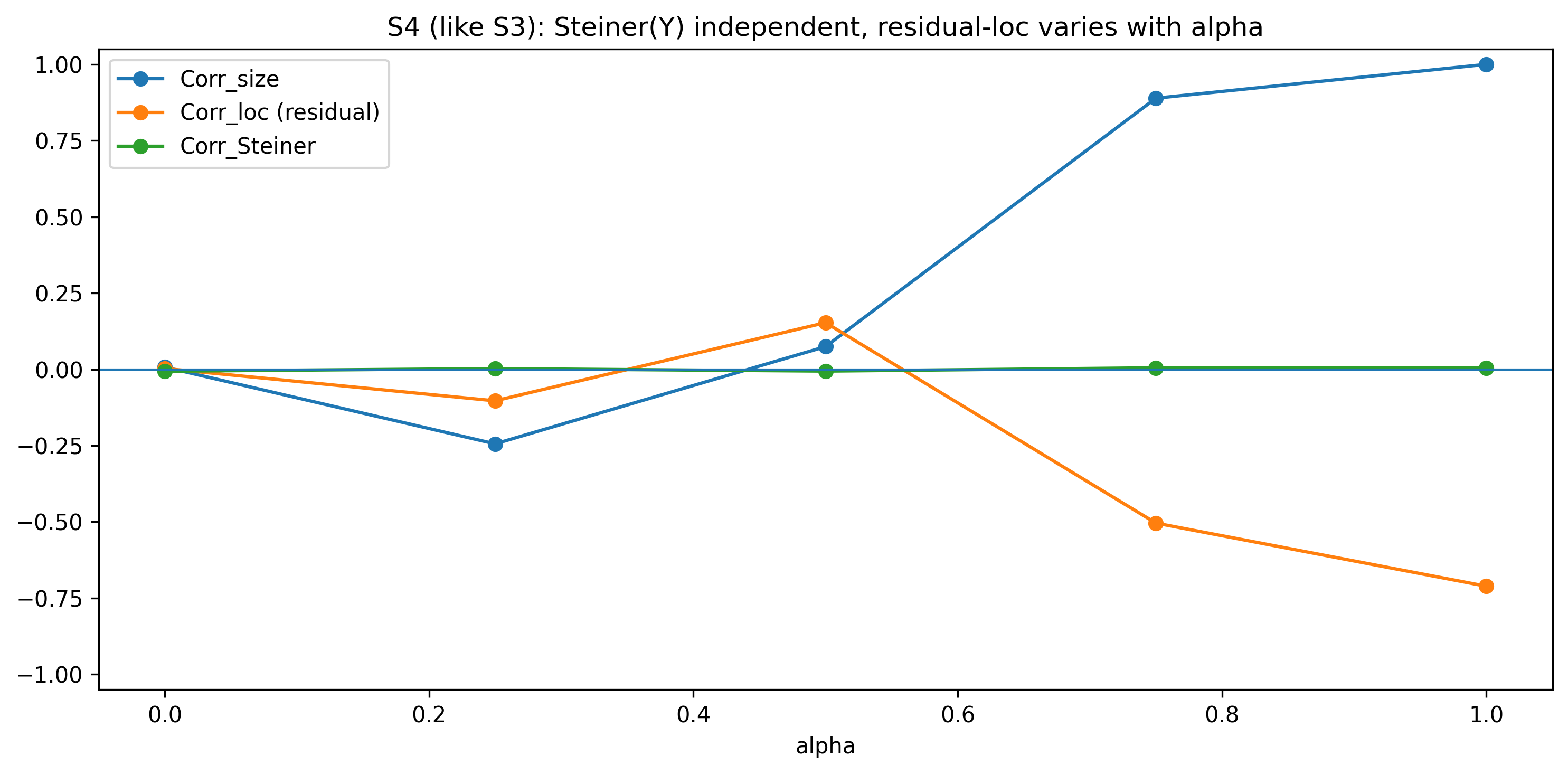}
	\caption{Scenario S4 (sensitivity analysis).
		The Steiner-point correlation remains close to zero for all $\alpha$,
		while the residual location correlation varies continuously, revealing
		directional location dependence invisible to Steiner-based summaries.}
	\label{fig:sim_S4_triangles}
\end{figure}

Across all values of $\alpha$, the Steiner-point correlation remains close
to zero.
In contrast, both the size correlation and the residual location
correlation vary markedly with $\alpha$, and do so in qualitatively
different ways.
This confirms that the proposed residual location index captures
directional dependence that cannot be reduced to, nor recovered from,
Steiner-point centering.

For completeness, Table~\ref{tab:sim-summary} reports mean and standard
deviation over repeated replications; the qualitative conclusions are
unchanged.

\begin{table}[ht]
	\centering
	\caption{Monte--Carlo mean (standard deviation) of correlation estimates over $R=200$ replications.}
	\label{tab:sim-summary}
	\begin{tabular}{c c c c}
		\hline
		Scenario & Corr$_{\mathrm{size}}$ & Corr$_{\mathrm{loc(res)}}$ & Corr$_{\mathrm{Steiner}}$ \\
		\hline
		S1 & $1.00\,(0.00)$ & $-0.98\,(0.02)$ & $-1.00\,(0.00)$ \\
		S2 & $0.01\,(0.05)$ & $0.28\,(0.07)$ & $0.99\,(0.01)$ \\
		S3 & $1.00\,(0.00)$ & $0.73\,(0.06)$ & $0.00\,(0.04)$ \\
		\hline
	\end{tabular}
\end{table}

\section{Applications}
\label{sec:appls}

This section illustrates how the proposed even--odd support--function
framework connects with existing methodologies and enables new modeling
perspectives for dependent set-valued data.
The emphasis is conceptual rather than algorithmic: each application
highlights how separating size and location components yields clearer
interpretation and cleaner asymptotic theory.
Additional technical details are deliberately kept to a minimum.

\subsection{Interval-valued regression revisited via support losses}
\label{sec:interval-reg}

Regression models for interval-valued responses are classically based on
a midpoint--radius decomposition, where the center and the range are
modeled separately.
This paradigm originates in symbolic data analysis
\cite{billard2007symbolic} and leads to the centre--range regression model
\cite{neto2008centre,neto2010constrained}.

Let $Y_i=[C_i-R_i,\ C_i+R_i]$ be an interval-valued response and
$x_i\in\mathbb R^p$ covariates.
Consider predictors of the form
\[
\widehat Y_i=[x_i^\top\beta - r_i,\ x_i^\top\beta + r_i],
\qquad r_i:=|x_i^\top\gamma|,
\]
or more generally $r_i$ given by a nonnegative link function.
Define the support-based quadratic loss
\[
\mathcal L_n(\beta,\gamma)
:=\frac{1}{n}\sum_{i=1}^n
\int_{\mathbb S^0}\!
\big(h_{Y_i}(u)-h_{\widehat Y_i}(u)\big)^2\,d\sigma(u),
\quad
\mathbb S^0=\{\pm1\},
\quad
d\sigma=\tfrac12(\delta_{+1}+\delta_{-1}).
\]
This loss is the one-dimensional specialization of the general
$L^2(\sigma)$ support risk studied in this paper.

\begin{proposition}[Midpoint--radius regression as an exact optimizer]
	\label{prop:interval-reg}
	For interval responses,
	$h_{Y_i}(u)=u\,C_i+R_i$ and
	$h_{\widehat Y_i}(u)=u\,x_i^\top\beta+r_i$.
	Consequently,
	\[
	\mathcal L_n(\beta,\gamma)
	=\frac{1}{n}\sum_{i=1}^n
	\Big[(C_i-x_i^\top\beta)^2+(R_i-|x_i^\top\gamma|)^2\Big].
	\]
	If $r_i=x_i^\top\gamma$ with $x_i^\top\gamma\ge0$, minimization of
	$\mathcal L_n$ decouples into two least-squares problems for the midpoint
	and the radius.
\end{proposition}

\begin{proof}
	On $\mathbb S^0$, the support function decomposes as
	$h_Y(u)=W_Y+u\,C_Y$ with $W_Y=R$ and $C_Y=C$.
	The $L^2(\sigma)$ loss therefore splits orthogonally into even and odd
	parts, yielding the stated decomposition.
\end{proof}

Proposition~\ref{prop:interval-reg} shows that the classical
midpoint--radius regression is the \emph{exact} minimizer of a natural
support-based risk.
In one dimension, the odd component coincides with the linear
(Steiner-induced) term, so that no residual location component remains.
This explains why midpoint--radius regression fully captures location
effects for intervals, while higher-dimensional random sets require the
richer even--odd framework developed here.
The extension to weakly dependent data follows directly from the
Hilbert-space limit theory established earlier.

\subsection{Robust linear programming with interval uncertainty}

Robust linear programming with interval coefficients is a classical topic
in optimization; see, e.g., Ben-Tal et al.~\cite{ben2002robust} and
Schrijver~\cite{Schrijver1998}.
From the random-set viewpoint, interval uncertainty corresponds to convex
set-valued coefficients whose geometry is naturally described by support
functions~\cite{Molchanov2005}.

Consider the linear program
\[
\min \{ c^\top x : Ax \le b,\ x \in \mathbb R^p \},
\]
where each coefficient row $(a_j,b_j)$ is interval-valued.
For a given decision vector $x$, define the residual interval
\[
R_j(x)
:= \{ a_j^\top x - b_j : a_j\in A_j,\ b_j\in B_j \}
=[C_j(x)-S_j(x),\, C_j(x)+S_j(x)].
\]
Here $C_j(x)$ represents the systematic shift (location), while $S_j(x)$
quantifies uncertainty inflation (size).

The classical robust constraint is equivalent to $C_j(x)+S_j(x)\le0$.
This recovers the standard interval-robust reformulation, but the
support-function representation highlights an explicit separation
between location and size effects.
Unlike point-valued summaries such as the Steiner point, the support
function preserves directional information and remains informative when
uncertainties are dependent or non-axis-aligned.
Such dependence-aware formulations fit naturally within the random-set
framework and extend beyond independent interval bounds.

\subsection{An open problem: interventions from dependent set-valued signals}

We conclude with an open-style example illustrating structural advantages
of the even--odd decomposition.
Suppose that at each time $t$ one observes a convex set-valued signal
$X_t\subset\mathbb R^d$, representing, for instance, a spatial risk or
uncertainty region.
The process $(X_t)$ is weakly dependent and described through its support
function.

At time $t$, a decision-maker selects a convex intervention set $U_t$
subject to a size budget, aiming to minimize a quadratic support loss at
time $t+1$,
\[
\mathcal R_t(U_t)
:= \mathbb E\| h_{X_{t+1}}-h_{U_t}\|_{2,\sigma}^2,
\qquad
\int_{\mathbb S^{d-1}} W_{U_t}(u)\,d\sigma(u)\le B.
\]
By the even--odd decomposition,
\[
\mathcal R_t(U_t)
=
\mathbb E\|C_{X_{t+1}}-C_{U_t}\|_{2,\sigma}^2
+
\mathbb E\|W_{X_{t+1}}-W_{U_t}\|_{2,\sigma}^2,
\]
so that location tracking and size allocation decouple at the level of
the objective.
This separation is obscured by scalar summaries such as Aumann
expectations or Steiner points, which mix the two effects.

We do not pursue a full solution here.
The purpose of this example is to illustrate how the proposed framework
naturally yields interpretable subproblems and admits dependence-aware
uncertainty quantification via the long-run covariance operator developed
earlier.
Developing complete stochastic control or optimization schemes in this
setting remains an open problem.

\medskip
\noindent\textbf{Link with simulations.}
The simulation study in Section~\ref{sec:simulation} complements these
applications by demonstrating that the residual location component
captures directional dependence effects that are invisible to
Steiner-point--based summaries, even in simple geometric settings.

\section{Conclusion and Future Work}

This paper develops a unified framework for dependence and asymptotic
analysis of set-valued stochastic processes by exploiting the even--odd
(size--location) decomposition of support functions.
By introducing componentwise covariance and correlation structures,
together with a compatible maximal correlation coefficient $\rho_{\max}$,
we obtain a dependence theory that remains informative under weak
dependence, avoids selection-based degeneracies, and is intrinsically
adapted to the geometry of random sets.
Within this framework, we establish a full range of classical limit
theorems, including laws of large numbers, Hilbert-space central and
functional central limit theorems, and laws of the iterated logarithm for
$\rho_{\max}$-mixing set-valued processes.

A key conceptual outcome is the strict separation between size-driven
uncertainty and location-driven drift.
This resolves structural limitations of existing approaches, where these
two effects are inevitably conflated.
In particular, the proposed framework remains non-degenerate when the
size component vanishes, and it reveals directional dependence phenomena
that are invisible to scalar-valued mixing coefficients and
Steiner-point--based summaries.

\medskip
\noindent\textbf{Future directions.}
The present work opens several avenues for further research.

\begin{itemize}
	\item \textbf{Refined dependence notions.}
	The existence of nontrivial dependence at the level of support
	functions, despite vanishing scalar correlations of all one-dimensional
	projections, suggests new classes of weak dependence for random sets.
	Developing other corr-based mixing coefficients and identifying optimal
	mixing conditions under which the same limit theorems hold remain open
	problems.
	
	\item \textbf{Geometry-driven limit theory.}
	The block-diagonal structure of long-run covariance operators induced
	by the even--odd decomposition points to deeper interactions between
	dependence and geometry.
	Systematic studies of geometric inequalities, affine invariance, and
	concentration phenomena for set-valued covariance and correlation are
	promising directions.
	
	\item \textbf{Statistical inference.}
	The separation of size and location effects naturally leads to new
	inferential questions, including testing for directional persistence
	versus pure amplitude variability and constructing dependence-robust
	estimators of long-run covariance operators for random sets.
	
	\item \textbf{Optimization and data-driven extensions.}
	The framework provides a principled foundation for optimization and
	decision problems driven by dependent set-valued signals, as well as a
	natural interface with high-dimensional and learning-based settings.
	Developing dependence-aware optimization and control methodologies for
	random sets remains largely unexplored.
\end{itemize}

Overall, the results indicate that dependence for set-valued processes is
not merely a technical extension of scalar theory, but a source of
genuinely new probabilistic structure.
By reconciling geometry, dependence, and asymptotics within a single
framework, this work lays the groundwork for a systematic theory of
weakly dependent random sets.

\section*{Acknowledgements}

The author gratefully acknowledges helpful discussions with colleagues on set-valued stochastic processes and weak dependence.
The author also thanks the anonymous referees for their careful reading and constructive comments, which helped improve the clarity and scope of the paper.

\section*{Data Availability Statement}

No data were generated or analyzed in this study.

\section*{Code Availability}

The simulation code supporting the findings of this study is available from the author upon reasonable request.

\section*{Use of Artificial Intelligence}

Generative artificial intelligence tools were used for language editing and formatting assistance.
The scientific content, mathematical results, and conclusions were produced solely by the author.

\section*{Author Contributions}

The author solely developed the theory, proofs, simulations, and writing of this manuscript.

\section*{Funding}

This research received no external funding.
\section*{Competing Interests}
The author declares that there are no competing interests.

\newpage
\appendix
\section{Consistency and convergence rates of simulation-based estimators}
\label{app:Consistency_simulation}

This appendix collects standard consistency and convergence results for the
Monte--Carlo plug-in estimators used in Section~\ref{sec:simulation}. These results are
included for completeness. They rely on classical laws of large numbers and
moment bounds for weakly dependent processes and do not constitute independent
contributions beyond the $L^2(\sigma)$ framework developed in the main text.

\subsection{Setting and assumptions}

Let $(X_i,Y_i)_{i\in\mathbb Z}$ be a weakly stationary sequence of convex compact
random sets in $\mathbb R^d$. For $\star\in\{\mathrm{size},\mathrm{loc}\}$, recall
\[
\Phi_X^{(\mathrm{size})}(u)=W_X(u),\qquad
\Phi_X^{(\mathrm{loc})}(u)=C_X(u),
\]
and define analogously $\Phi_Y^{(\star)}$.

For $u\in\mathbb S^{d-1}$, set
\[
g_\star(u)
:=\Cov\!\big(\Phi_X^{(\star)}(u),\Phi_Y^{(\star)}(u)\big),\qquad
\Cov_\star(X,Y)=\int_{\mathbb S^{d-1}} g_\star(u)\,d\sigma(u).
\]

We work under the following assumptions.

\begin{itemize}
	\item[(A1)] \emph{Integrated second moments:}
	\[
	\int_{\mathbb S^{d-1}}\E\,h_{X_0}(u)^2\,d\sigma(u)<\infty,
	\qquad
	\int_{\mathbb S^{d-1}}\E\,h_{Y_0}(u)^2\,d\sigma(u)<\infty.
	\]
	
	\item[(A2)] \emph{Weak dependence:}
	The sequence $(X_i,Y_i)$ satisfies a weak-dependence condition ensuring LLN
	and CLT for second-order statistics, e.g.\ $\rho$-mixing with
	$\sum_{k\ge1}\rho(k)<\infty$.
	
	\item[(A3)] \emph{Random directions:}
	$u_1,\dots,u_M\stackrel{\text{i.i.d.}}{\sim}\mathrm{Unif}(\mathbb S^{d-1})$,
	independent of the data.
\end{itemize}

\subsection{Per-direction covariance estimation}

For a fixed direction $u\in\mathbb S^{d-1}$, define
\[
\widehat{\Cov}_n^{(\star)}(u)
:=\frac{1}{n-1}\sum_{i=1}^n
\big(\Phi_{X_i}^{(\star)}(u)-\overline\Phi_n^{(\star)}(u)\big)
\big(\Phi_{Y_i}^{(\star)}(u)-\overline\Psi_n^{(\star)}(u)\big),
\]
where $\overline\Phi_n^{(\star)}(u)=n^{-1}\sum_i\Phi_{X_i}^{(\star)}(u)$ and
$\overline\Psi_n^{(\star)}(u)$ is defined similarly.

\begin{lemma}[Per-direction mean-square error]
	\label{lem:per-u-app}
	Assume \textnormal{(A1)--(A3)}. Suppose in addition that one of the following holds:
	\begin{itemize}
		\item[(A1$^+$)] $\displaystyle
		\int_{\mathbb S^{d-1}}\E\,h_{X_0}(u)^4\,d\sigma(u)<\infty$ and similarly for $Y_0$;
		
		\item[(A1$^{++}$)] $\displaystyle
		\int_{\mathbb S^{d-1}}\E\,h_{X_0}(u)^{2+\delta}\,d\sigma(u)<\infty$ for some
		$\delta>0$, together with a Davydov--Rio type inequality for the dependence.
	\end{itemize}
	Then, for $\star\in\{\mathrm{size},\mathrm{loc}\}$,
	\[
	\int_{\mathbb S^{d-1}}
	\E\big(\widehat{\Cov}_n^{(\star)}(u)-g_\star(u)\big)^2\,d\sigma(u)
	\ \le\ \frac{C_\star}{n},
	\]
	for some finite constant $C_\star$.
\end{lemma}

\begin{proof}
	This is a standard second-moment bound for sample covariance estimators under
	weak dependence. Writing the Bessel-corrected covariance in terms of centered
	products and separating the empirical mean terms yields a decomposition into
	a variance term of order $n^{-1}$ and a bias term of order $n^{-2}$. Integrated
	moment bounds follow from \textnormal{(A1$^+$)} or \textnormal{(A1$^{++}$)} by
	Cauchy--Schwarz and Davydov--Rio inequalities. The detailed argument coincides
	with the proof given in the main manuscript and is omitted here for brevity.
\end{proof}

\subsection{Integrated estimators and Monte--Carlo error}

The Monte--Carlo estimator based on $M$ random directions is
\[
\widehat{\Cov}_\star^{(M,n)}(X,Y)
:=\frac1M\sum_{j=1}^M \widehat{\Cov}_n^{(\star)}(u_j).
\]

\begin{theorem}[Consistency and convergence rates]
	\label{thm:consistency-rate-app}
	Assume \textnormal{(A1)--(A3)}. Then, as $M,n\to\infty$,
	\[
	\widehat{\Cov}_\star^{(M,n)}(X,Y)
	\ \xrightarrow{\ \mathbb P\ }\ 
	\Cov_\star(X,Y),
	\qquad
	\star\in\{\mathrm{size},\mathrm{loc},\mathrm{tot}\}.
	\]
	If, in addition, \textnormal{(A1$^+$)} holds, then there exist constants
	$\kappa_1,\kappa_2<\infty$ such that
	\[
	\Var\!\big(\widehat{\Cov}_\star^{(M,n)}(X,Y)\big)
	\ \le\ \frac{\kappa_1}{n}\ +\ \frac{\kappa_2}{M}.
	\]
\end{theorem}

\begin{proof}
	Decompose
	\[
	\widehat{\Cov}_\star^{(M,n)}-\Cov_\star
	=\frac1M\sum_{j=1}^M\big(\widehat{\Cov}_n^{(\star)}(u_j)-g_\star(u_j)\big)
	+\Big(\frac1M\sum_{j=1}^M g_\star(u_j)-\int g_\star\,d\sigma\Big).
	\]
	The second term converges to zero by the law of large numbers for i.i.d.\
	directions. The first term converges by Lemma~\ref{lem:per-u-app}. Variance bounds
	follow from the law of total variance, separating Monte--Carlo error
	($O(M^{-1})$) from statistical error ($O(n^{-1})$).
\end{proof}

\begin{corollary}[Consistency of correlation estimators]
	Assume $\Var_\star(X)>0$ and $\Var_\star(Y)>0$. Then
	\[
	\widehat{\Corr}_\star^{(M,n)}(X,Y)
	\ \xrightarrow{\ \mathbb P\ }\ 
	\Corr_\star(X,Y),
	\qquad
	\star\in\{\mathrm{size},\mathrm{loc},\mathrm{tot}\}.
	\]
\end{corollary}

\begin{proof}
	Apply Theorem~\ref{thm:consistency-rate-app} to the covariance and variance
	estimators and invoke Slutsky’s theorem.
\end{proof}

\begin{remark}
	Choosing $M\simeq n$ balances Monte--Carlo and statistical errors and yields
	mean squared error of order $n^{-1}$. These rates are standard and comparable
	to those of classical integrated covariance estimators.
\end{remark}

\newpage
\bibliographystyle{plain}      
\bibliography{References_R1}   

@book{li2013limit,
	title={Limit theorems and applications of set-valued and fuzzy set-valued random variables},
	author={Li, Shoumei and Ogura, Yukio and Kreinovich, Vladik},
	volume={43},
	year={2013},
	publisher={Springer Science \& Business Media}
}

@article{artstein1975strong,
	title={A strong law of large numbers for random compact sets},
	author={Artstein, Zvi and Vitale, Richard A},
	journal={The Annals of Probability},
	pages={879--882},
	year={1975},
	publisher={JSTOR}
}

@article{chen2015strong,
	title={Strong law of large numbers for upper set-valued and fuzzy-set valued probability},
	author={Chen, Zengjing and Lan, Yuting and Zong, Gaofeng},
	journal={Mathematical Control \& Related Fields},
	volume={5},
	number={3},
	pages={435},
	year={2015},
	publisher={American Institute of Mathematical Sciences}
}

@article{bradley2005basic,
	title={Basic properties of strong mixing conditions. A survey and some open questions},
	author={Bradley, Richard C},
	journal={Probability Surveys},
	volume={2},
	pages={107--144},
	year={2005},
	doi="10.1214/154957805100000104",
	publisher={Project Euclid}
}

@book{conway2019course,
	title={A course in functional analysis},
	author={Conway, John B},
	volume={96},
	year={2019},
	publisher={Springer}
}

@book{rio2017asymptotic,
	title={Asymptotic theory of weakly dependent random processes},
	author={Rio, Emmanuel and others},
	volume={80},
	year={2017},
	publisher={Springer}
}

@incollection{merlevede2009bernstein,
	title={Bernstein inequality and moderate deviations under strong mixing conditions},
	author={Merlev{\`e}de, Florence and Peligrad, Magda and Rio, Emmanuel},
	booktitle={High dimensional probability V: the Luminy volume},
	volume={5},
	pages={273--293},
	year={2009},
	publisher={Institute of Mathematical Statistics}
}

@book{Schneider2014,
	author    = {Schneider, Rolf},
	title     = {Convex Bodies: The Brunn--Minkowski Theory},
	edition   = {2nd expanded ed.},
	publisher = {Cambridge University Press},
	address   = {Cambridge},
	year      = {2014}
}

@book{Gardner2006,
	author    = {Gardner, Richard J.},
	title     = {Geometric Tomography},
	edition   = {2nd ed.},
	publisher = {Cambridge University Press},
	address   = {Cambridge},
	year      = {2006}
}

@article{neto2008centre,
	title={Centre and range method for fitting a linear regression model to symbolic interval data},
	author={Neto, Eufr{\'a}sio de A Lima and De Carvalho, Francisco de AT},
	journal={Computational Statistics \& Data Analysis},
	volume={52},
	number={3},
	pages={1500--1515},
	year={2008},
	publisher={Elsevier}
}

@article{neto2010constrained,
	title={Constrained linear regression models for symbolic interval-valued variables},
	author={Neto, Eufr{\'a}sio de A Lima and De Carvalho, Francisco de AT},
	journal={Computational Statistics \& Data Analysis},
	volume={54},
	number={2},
	pages={333--347},
	year={2010},
	publisher={Elsevier}
}

@misc{billard2007symbolic,
	title={Symbolic Data Analysis: Conceptual Statistics and Data Mining (Wiley Series in Computational Statistics)},
	author={Billard, Lynne and Diday, Edwin},
	year={2007},
	publisher={John Wiley \& Sons, Inc.}
}

@article{ben2002robust,
	title={Robust optimization--methodology and applications},
	author={Ben-Tal, Aharon and Nemirovski, Arkadi},
	journal={Mathematical programming},
	volume={92},
	number={3},
	pages={453--480},
	year={2002},
	publisher={Springer}
}

@book{Schrijver1998,
	author    = {Alexander Schrijver},
	title     = {Theory of Linear and Integer Programming},
	publisher = {John Wiley \& Sons},
	year      = {1998}
}

@book{Molchanov2005,
	author    = {Ilya Molchanov},
	title     = {Theory of Random Sets},
	publisher = {Springer},
	year      = {2005}
}

@article{tuyen2026strong,
	title={On strong law of large numbers for weakly stationary $\varphi$-mixing set-valued random variable sequences}, 
	author={Luc Tri Tuyen},
	year={2026},
	journal = {arXiv preprint arXiv:2601.09197},
	archivePrefix={arXiv},
	primaryClass={math.PR},
	url={https://arxiv.org/abs/2601.09197}, 
}
\end{document}